\documentclass{amsart}
\usepackage{graphicx}
\usepackage{amsmath}
\usepackage{amsfonts}
\usepackage{amssymb}
\usepackage{enumitem}
\usepackage{ifpdf}
\providecommand{\U}[1]{\protect\rule{.1in}{.1in}}
\newtheorem{theorem}{Theorem}

\newtheorem{corollary}{Corollary}

\newtheorem{definition}{Definition}
\newtheorem{example}{Example}

\newtheorem{lemma}{Lemma}

\newtheorem{proposition}{Proposition}
\newtheorem{remark}{Remark}

\makeatletter \@namedef{subjclassname@2020}{\textup{2020}
Mathematics Subject Classification} \makeatother

\begin{document}
\title[Poincar\'e inequalities and doubling]
{Poincar\'e inequalities, logarithmic growth and doubling}
\author{Joaquim Mart\'{\i}n}
\address{Department of Mathematics, Universitat Aut\`onoma de
Barcelona, 08193 Bellaterra (Barcelona), Spain}
\email{Joaquin.Martin@uab.cat}
\urladdr{https://orcid.org/0000-0002-7467-787X}
\thanks{Partially supported by Grants PID2024-160507NB-I00 and PID2024-155917NB-I00
funded by MCIN/AEI/10.13039/501100011033.} \subjclass[2020]{Primary
30L15; Secondary 46E30, 46E36} \keywords{Poincar\'e inequalities,
doubling measures, rearrangement-invariant spaces, log-concave
measures, Gaussian measures, metric measure spaces}

\begin{abstract}

We investigate when a Poincar\'e inequality forces the underlying
measure to be doubling, with particular attention to the space \(X\)
controlling the oscillation. We identify \(L\log L\) as the sharp
threshold: under a natural local growth condition on the measure,
doubling is forced precisely when \(L\log L(0,1)\not\hookrightarrow
X\). We also obtain a doubling criterion under a weaker geometric
assumption when the oscillation norm is sufficiently stronger than
\(L\log L\).

We show that this threshold is sharp in every fixed Euclidean
dimension, using a dimension-free local \(L\log L\)-Poincar\'e
inequality for log-concave measures and non-doubling Gaussian
examples. We then turn to infinite-dimensional spaces, where
doubling is impossible for the measures under consideration.
Nevertheless, the \(L\log L\)-Poincar\'e inequality continues to
hold, and we establish sharp characterizations for infinite product
models, Gaussian measures on Hilbert spaces under power-type
small-ball estimates, and Wiener measure with the uniform norm.
\end{abstract}

\maketitle
\tableofcontents

\section{Introduction}
\label{sec:introduction}

Throughout the paper, \((\Omega,d,\mu)\) denotes a metric measure
space, where \(\mu\) is a nonatomic Borel measure satisfying
\(0<\mu(B)<\infty\) for every ball \(B\subset\Omega\). All balls are
understood to be open. For a ball \(B=B(x,r)\), we write \(r_B=r\),
\(\sigma B=B(x,\sigma r)\), and \(f_{B,\mu}=\mu(B)^{-1}\int_B
f\,d\mu\).

The doubling condition is one of the basic structural assumptions in
the analysis of metric measure spaces. It requires a constant
\(C_D\geq1\) such that \(\mu(B(x,2r))\leq C_D\mu(B(x,r))\) for every
\(x\in\Omega\) and \(r>0\), and provides many of the standard tools
of analysis on metric spaces (see
\cite{HeinonenKoskela,HajlaszKoskela,FPW,HKST} and the references
therein).

Another basic tool is the family of Poincar\'e inequalities. For
\(1\leq p,q<\infty\), a \((q,p)\)-Poincar\'e inequality takes the
form
\begin{equation}
\label{eq:qp-Poincare}
\left(\frac1{\mu(B)}\int_B|f-f_{B,\mu}|^q\,d\mu\right)^{1/q} \leq C
r_B \left(\frac1{\mu(\sigma B)} \int_{\sigma
B}g^p\,d\mu\right)^{1/p},
\end{equation}
for every ball \(B\), every locally integrable function \(f\), and
every upper gradient \(g\) of \(f\), with constants \(C>0\) and
\(\sigma\geq1\) independent of \(B\), \(f\), and \(g\). Precise
definitions and conventions are given in Section~\ref{sec:pre}.

The interaction between Poincar\'e inequalities and the doubling
condition is fundamental to analysis on metric measure spaces (see,
for example,
\cite{Ha,HajlaszKoskela,Sh,HKST,AH,AlvaradoGorkaHajlasz,FPW,
ErikssonBique,HeikkinenKarak} and the references therein). The
relation works in both directions. On a doubling space, a
\((1,1)\)-Poincar\'e inequality improves to a \((q,1)\)-Poincar\'e
inequality for some \(q>1\) (see
\cite[Theorem~5.1]{HajlaszKoskela}). Conversely, a
\((q,1)\)-Poincar\'e inequality with \(q>1\) forces the measure to
be doubling (see \cite{AH}); related weak Sobolev inequalities also
force doubling (see \cite{KorobenkoMaldonadoRios}). Thus a fixed
power improvement on the oscillation side is sufficient to force
doubling.

It is therefore natural to consider improvements weaker than any
fixed power gain, as in the Orlicz setting studied in
\cite{Korobenko,HeikkinenKarak}. Such improvements also arise in the
regularity theory of infinitely degenerate elliptic equations, where
Orlicz--Sobolev inequalities are used to obtain regularity in
non-doubling geometries. These inequalities involve adapted
gradients and may require replacing the metric radius by a larger
quantity, called a superradius (see \cite{KRSSMemoirs,Korobenko} and
the references therein).

In the Poincar\'e setting, this raises the following question:
\begin{center}
\emph{How much stronger than \(L^1\) must the norm controlling the
oscillation in a Poincar\'e inequality be in order to force the
underlying measure to be doubling? Is there a critical threshold
separating the doubling and non-doubling regimes?}
\end{center}

Logarithmic and iterated logarithmic improvements have already been
studied in the Orlicz setting (see \cite{Korobenko,HeikkinenKarak}).
This raises the question of whether the relevant threshold is
specific to the Orlicz structure, or is instead governed by a more
general property of the target norm. To address this question, we
work with rearrangement-invariant Banach function spaces (r.i.\
spaces for short).\footnote{A Banach function space is rearrangement
invariant if equimeasurable functions have the same norm, that is,
\(\|f\|_X=\|g\|_X\) whenever \(f\) and \(g\) have the same
distribution function.} This class includes many classical function
spaces, such as Lebesgue, Orlicz, Lorentz, Lorentz--Zygmund, and
Marcinkiewicz spaces, together with their logarithmic refinements
(see \cite{BennettSharpley,KreinPetuninSemenov}).

In this setting, the Poincar\'e inequalities considered in this
paper take the following form\footnote{We refer to
Section~\ref{sec:pre} for the relevant notation and further details
on rearrangements and normalized r.i.\ norms.}

\begin{definition}
\label{def:PI} Let \((\Omega,d,\mu)\) be a metric measure space such
that \(0<\mu(B)<\infty\) for every ball \(B\subset\Omega\), and let
\(X\) be an r.i.\ space on \((0,1)\). For every ball \(B\), set
\[
\mu_B(E)=\frac{\mu(E\cap B)}{\mu(B)}.
\]
We denote by \((\,\cdot\,)_{\mu_B}^*\) the decreasing rearrangement
with respect to \(\mu_B\).

\begin{enumerate}
\item[\rm (i)]
We say that \((\Omega,d,\mu)\) supports an \((X,L^1)\)-Poincar\'e
inequality if there exist constants \(C_P>0\) and \(\sigma\geq1\)
such that
\begin{equation}
\label{eq:intro-PI} \big\|(f-f_{B,\mu})_{\mu_B}^*\big\|_X \leq C_P
r_B\frac1{\mu(\sigma B)} \int_{\sigma B}g\,d\mu
\end{equation}
for every ball \(B\), every locally integrable function \(f\), and
every upper gradient \(g\) of \(f\).

\item[\rm (ii)]
We say that \((\Omega,d,\mu)\) supports the testing
\((X,L^1)\)-Poincar\'e inequality if there exist constants \(C_P>0\)
and \(\sigma\geq1\) such that
\begin{equation}
\label{eq:PI-B} \big\|(f-f_{B,\mu})_{\mu_B}^*\big\|_X \leq C_P
r_B\frac1{\mu(B)} \int_{\sigma B}g\,d\mu
\end{equation}
for every ball \(B\), every bounded Lipschitz function \(f\), and
every upper gradient \(g\) of \(f\).
\end{enumerate}
\end{definition}

The \((X,L^1)\)-Poincar\'e inequality (\ref{eq:intro-PI}) implies
the testing \((X,L^1)\)-Poincar\'e inequality (\ref{eq:PI-B}).

At this point, it is useful to notice that doubling itself already
gives a uniform control on the lower local dimension. Indeed, if
\(\mu\) is doubling with constant \(C_D\), then a standard iteration
of the doubling inequality gives a polynomial lower bound for the
measure of small balls (see \cite{HeinonenLectures}). In particular,
\[
\underline d_\mu(x) =\liminf_{r\downarrow0}
\frac{\log\mu(B(x,r))}{\log r} \leq \log_2 C_D \qquad\text{for every
}x\in\Omega.
\]
Thus uniformly bounded lower local dimension is a necessary
consequence of doubling. It is therefore natural to begin with
measures for which this a priori control is available.

The functional side of the problem is encoded by
\[
U_X(t)=t\|\chi_{(0,1/t)}\|_X,\qquad t\geq1.
\]
This is the ratio between the \(X\)-norm and the \(L^1\)-norm of a
characteristic function of a set of measure \(1/t\). Thus
\(U_{L^1}(t)=1\), \(U_{L^q}(t)=t^{1-1/q}\) for \(q>1\), whereas
\[
U_{L\log L}(t)\simeq\log t,\qquad t\to\infty.
\]
Consequently, the behavior of
\[
\frac{U_X(t)}{\log t}
\]
measures the position of \(X\) relative to the \(L\log L\)
endpoint.\footnote{\(A\preceq B\) means \(A\leq CB\), with \(C\)
independent of the variables under consideration, and \(A\simeq B\)
means that both \(A\preceq B\) and \(B\preceq A\).}

Our first main result identifies \(L\log L\) as the exact threshold
for forcing doubling.

\begin{theorem}
\label{thm:finite-sharp} Let \(X\) be an r.i.\ space on \((0,1)\).
The following statements are equivalent:
\begin{enumerate}
\item[\rm (i)]
\(L\log L(0,1)\not\hookrightarrow X\).

\item[\rm (ii)]
Every metric measure space for which \(\underline d_\mu\in
L^\infty(\Omega,\mu)\) and which supports the testing
\((X,L^1)\)-Poincar\'e inequality (\ref{eq:PI-B}) has doubling
measure, with doubling constant depending only on \(C_P\), \(X\),
and \(\|\underline d_\mu\|_{L^\infty(\Omega,\mu)}\).

\item[\rm (iii)]
\[
\sup_{t\geq2}\frac{U_X(t)}{\log t}=\infty.
\]
\end{enumerate}
Moreover, for every fixed \(n\geq1\), these conditions are
equivalent to the following statement: every absolutely continuous
measure \(d\mu=w\,dx\) on \(\mathbb R^n\), equipped with the
Euclidean distance, with \(w\in L^1_{\rm loc}(\mathbb R^n)\) and
\(w>0\) almost everywhere, which satisfies (\ref{eq:PI-B}), is
doubling.
\end{theorem}

The dimensional assumption in {\rm (ii)} is automatically satisfied
when the upper Hausdorff dimension
\[
\dim_H^+\mu =\inf\{\dim_H F:F\subset\Omega\text{ Borel}, \
\mu(\Omega\setminus F)=0\}
\]
is finite: in that case \(\underline d_\mu\leq\dim_H^+\mu\) almost
everywhere, and hence on a dense set (see Lemma~\ref{lem:dim}). In
particular, this holds whenever \(\dim_H\Omega<\infty\). The
dimensional assumption concerns the measure rather than the ambient
space. Indeed, \(\dim_H^+\mu\) may be finite even when
\(\dim_H\Omega=\infty\), so the theorem also applies in such
settings.

The endpoint is sharp already in every fixed Euclidean dimension. We
prove a dimension-free local \(L\log L\)-Poincar\'e inequality for
log-concave measures; Gaussian measures then provide non-doubling
examples at the endpoint (see Corollary~\ref{cor:Gaussian}).

Theorem~\ref{thm:finite-sharp} also applies to logarithmic
improvements over \(L\log L\) that are weaker than any fixed
additional logarithmic power. For example, let \(X=L^\Phi(0,1)\) be
the Orlicz space associated with a Young function \(\Phi\)
satisfying
\[
\Phi(t)\simeq t\left(\prod_{j=1}^{k-1}\log^{(j)}t\right)
(\log^{(k)}t)^\alpha, \qquad t\to\infty,
\]
with \(\log^{(1)}t=\log(e+t)\),
\(\log^{(j)}t=\log(e+\log^{(j-1)}t)\) for \(j\geq2\), and the empty
product understood as \(1\). Then
\[
\frac{U_X(t)}{\log t}\to\infty
\]
when \(k=1\) and \(\alpha>1\), and, for every \(k\geq2\), whenever
\(\alpha>0\). In particular, for \(k\geq2\) this includes the range
\(0<\alpha\leq1\).

The range \(\alpha>1\) is covered in
\cite[Theorem~1.19]{HeikkinenKarak} and by our earlier criterion for
r.i.\ spaces (see \cite[Theorem~8 and
Example~15(iii)]{MartinOrtizPAMS}). Thus, for \(k\geq2\),
Theorem~\ref{thm:finite-sharp} adds the range \(0<\alpha\leq1\)
under the stated assumption on the lower local dimension.

The intermediate logarithmic factors are not required either. For
every \(m\geq3\) and \(\alpha>0\), the Young function
\[
\Phi(t)\simeq t\log^{(1)}t\,(\log^{(m)}t)^\alpha, \qquad t\to\infty,
\]
satisfies
\[
\frac{U_{L^\Phi}(t)}{\log t}
\simeq(\log^{(m)}t)^\alpha\longrightarrow\infty.
\]
These examples are not covered by the criteria cited above, but
Theorem~\ref{thm:finite-sharp} applies under the same dimensional
assumption.

The same quotient \(U_X(t)/\log t\) also governs how much a priori
geometric information is needed. The lower local dimension captures
polynomial small-ball behavior. To go beyond this regime, we use the
notion of local order introduced in the framework of scales by
Helfter (see \cite{Helfter}). The lower local order is defined by
\[
\underline{\operatorname{ord}}_{\mathrm{loc}}\mu(x)
=\liminf_{r\downarrow0} \frac{\log\log(1/\mu(B(x,r)))}{\log(1/r)}.
\]
Finite lower local dimension implies lower local order zero, whereas
finite lower local order may still hold when the lower local
dimension is infinite. Thus the latter provides strictly weaker a
priori geometric information.

Our second theorem shows that weaker geometric information is
sufficient when the functional condition is strengthened from
unboundedness of \(U_X(t)/\log t\) to divergence to infinity. Under
this assumption, the testing Poincar\'e inequality makes essential
boundedness of the lower local order equivalent to doubling.
\begin{theorem}
\label{thm:SB} Let \(X\) be an r.i.\ space on \((0,1)\) such that
\(U_X(t)/\log t\to\infty\) as \(t\to\infty\), and let
\((\Omega,d,\mu)\) support the testing \((X,L^1)\)-Poincar\'e
inequality (\ref{eq:PI-B}). The following statements are equivalent:
\begin{enumerate}
\item[\rm (i)] \(\mu\) is doubling.
\item[\rm (ii)]
\(\underline d_\mu\in L^\infty(\Omega,\mu)\).
\item[\rm (iii)]
\(\underline{\operatorname{ord}}_{\rm loc}\mu \in
L^\infty(\Omega,\mu)\).
\end{enumerate}
Whenever these conditions hold, the doubling constant depends only
on \(C_P\) and \(X\).
\end{theorem}

Whenever \(U_X(t)/\log t\) has a limit in \([0,\infty]\), the two
functional conditions coincide. All the iterated logarithmic
examples above satisfy the stronger condition, so
Theorem~\ref{thm:SB} applies to them under the weaker assumption of
finite lower local order on a dense set.

We next turn to infinite-dimensional settings. A doubling measure
with full support forces the underlying metric space to be
geometrically doubling, and hence to have finite Hausdorff dimension
(see \cite{HeinonenLectures}). Thus the infinite-dimensional models
considered below cannot support a doubling measure with full
support. Nevertheless, the \(L\log L\)-Poincar\'e endpoint continues
to hold. We examine three settings in which this phenomenon can be
described sharply.

The first consists of infinite tensor products of one-dimensional
log-concave probability measures, equipped with weighted supremum
metrics. Under the condition on the weights stated in
Theorem~\ref{thm:product} below, the \((X,L^1)\)-Poincar\'e
inequality holds if and only if
\[
L\log L(0,1)\hookrightarrow X.
\]
This includes Hilbert cubes and infinite products of Gaussian and
exponential laws.

The second consists of centered Gaussian measures on separable
Hilbert spaces. We prove a dimension-free local \(L\log
L\)-Poincar\'e inequality (Theorem~\ref{thm:HilbertGauss}) and
obtain the same exact characterization under the two-sided
small-ball estimate
\[
-\log\gamma(B(0,r))\simeq r^{-\beta}
\]
for some \(\beta>0\) (see Corollary~\ref{cor:HilbertGauss}).
Brownian motion viewed as a Gaussian measure on \(L^2(0,1)\)
provides a concrete example.

Finally, for Wiener measure on \(C_0([0,1])\), equipped with the
uniform norm, we obtain the same exact endpoint (see
Corollary~\ref{cor:Wiener}). Here upper gradients are taken with
respect to the uniform metric.

Thus \(L\log L\) plays two complementary roles. When doubling is
geometrically possible, it is the critical functional threshold for
forcing doubling. In the infinite-dimensional models considered
above, where doubling with full support is impossible, the same
logarithmic endpoint remains valid and is sharp.

The paper is organized as follows. Section~\ref{sec:pre} introduces
the notation and collects the preliminary results on
rearrangement-invariant spaces and Poincar\'e inequalities.
Section~\ref{sec:preparatory} establishes the radial estimates and
their consequences for local doubling and small-ball growth,
together with the logarithmic endpoint for log-concave measures.
These results are then used in Section~\ref{sec:main-proofs} to
prove Theorems~\ref{thm:finite-sharp} and~\ref{thm:SB}. Finally,
Section~\ref{sec:infinite} establishes sharp logarithmic endpoints
for infinite product models, Gaussian measures on Hilbert spaces,
and Wiener measure with the uniform norm.

\section{Preliminaries}
\label{sec:pre}

\subsection{Background and notation}
\label{sec:background}

We keep the standing assumptions and the notation of
Section~\ref{sec:introduction}. No finiteness assumption is imposed
on \(\mu(\Omega)\), and we say that \(f\) is locally integrable if
it is integrable on every ball. For a ball \(B\), we write
\(\mu_B=\mu|_B/\mu(B)\).

For \(x\in\Omega\) and \(r>0\), set
\[
V_x(r)=\mu(B(x,r)), \qquad D_x(r)=\frac{V_x(2r)}{V_x(r)}, \qquad
\psi_x(r)=\log\frac1{V_x(r)}.
\]
We refer to \(\psi_x\) as the small-ball function at \(x\). Notice
that \(\mu\) is doubling if and only if
\[
\sup_{x\in\Omega,\ r>0}D_x(r)<\infty.
\]

A nonnegative Borel function \(g:\Omega\to[0,\infty]\) is an upper
gradient of a measurable function \(f:\Omega\to\mathbb R\) if
\[
|f(x)-f(y)|\leq\int_\gamma g\,ds
\]
for all \(x,y\in\Omega\) and every rectifiable curve \(\gamma\)
joining \(x\) and \(y\) (see \cite{HeinonenKoskela,HKST}). Throughout
the paper, upper gradients are understood in this sense.

Let \((E,\nu)\) be a finite nonatomic measure space. For a
measurable function \(f\) on \(E\), its decreasing rearrangement
with respect to \(\nu\) is
\[
f_\nu^*(s) = \inf\{t\geq0:\nu\{x\in E:|f(x)|>t\}\leq s\}, \qquad
0<s<\nu(E).
\]
When the underlying measure is clear, we omit the subscript. In
particular, for functions on \((0,1)\), \(f^*\) denotes the
decreasing rearrangement with respect to Lebesgue measure.

A Banach function space \(X\) on \((0,1)\), with the Fatou property,
is called rearrangement-invariant (r.i.) if
\[
\|f\|_X=\|g\|_X \qquad\text{whenever }f^*=g^*.
\]
As usual, r.i.\ spaces have the lattice property and satisfy
\[
L^\infty(0,1)\hookrightarrow X\hookrightarrow L^1(0,1).
\]
Throughout, \(X\) denotes a fixed r.i.\ space on \((0,1)\),
normalized by
\[
\|\chi_{(0,1)}\|_X=1.
\]
We refer to \cite{BennettSharpley,KreinPetuninSemenov} for the basic
theory of rearrangements and r.i.\ spaces.

\begin{definition}
\label{def:rirestricted} Let \(E\subset\Omega\) be measurable, with
\(0<\mu(E)<\infty\), and set
\[
\mu_E=\frac{\mu|_E}{\mu(E)}.
\]
We define \(X(E,\mu_E)\) by
\[
\|f\|_{X(E,\mu_E)} := \|f_{\mu_E}^*\|_X.
\]
Equivalently, writing \(f_\mu^*\) for the decreasing rearrangement
of \(f|_E\) with respect to \(\mu|_E\), we have
\[
f_{\mu_E}^*(s)=f_\mu^*(\mu(E)s), \qquad 0<s<1,
\]
and hence
\[
\|f\|_{X(E,\mu_E)} = \big\|f_\mu^*(\mu(E)\,\cdot)\big\|_X.
\]
We refer to this as the normalized \(X\)-norm on \(E\) (see
\cite[(2.13)--(2.14)]{CavaliereCianchi}).
\end{definition}

In particular, for a ball \(B\),
\[
\|f\|_{X(B,\mu_B)}=\|f_{\mu_B}^*\|_X.
\]
With this notation, the Poincar\'e inequality (\ref{eq:intro-PI})
and the testing condition (\ref{eq:PI-B}) introduced in the
Introduction are understood in terms of the normalized \(X\)-norm on
each ball.

For \(X=L^q(0,1)\), \(1\leq q<\infty\),
\[
\|f\|_{L^q(B,\mu_B)} = \left( \frac1{\mu(B)}\int_B|f|^q\,d\mu
\right)^{1/q},
\]
and (\ref{eq:intro-PI}) is the classical \((q,1)\)-Poincar\'e
inequality (\ref{eq:qp-Poincare}).

If \(X=L^\Phi(0,1)\) is the Orlicz space associated with a Young
function \(\Phi\), then
\[
\|h\|_{L^\Phi(B,\mu_B)} = \inf\left\{ \lambda>0: \frac1{\mu(B)}
\int_B \Phi\left(\frac{|h|}{\lambda}\right)\,d\mu \leq1 \right\}.
\]
We denote by \(L\log L\) the Orlicz space corresponding to
\begin{equation}
\label{eq:Phi0} \Phi_0(t)=\frac{t\log(e+t)}{\log(e+1)}, \qquad
t\geq0.
\end{equation}
This is the critical endpoint in the results below.

The fundamental function of \(X\) is defined by
\begin{equation*}
\phi_X(s)=\|\chi_{(0,s)}\|_X, \qquad 0<s\leq1. 
\end{equation*}
By our normalization, \(\phi_X(1)=1\), and we set \(\phi_X(0)=0\).
From Definition~\ref{def:rirestricted}, for every measurable set
\(E\subset B\),
\[
\|\chi_E\|_{X(B,\mu_B)} = \phi_X\left(\frac{\mu(E)}{\mu(B)}\right).
\]

We shall use the function
\[
U_X(t)=t\phi_X(1/t), \qquad t\geq1.
\]
Thus \(U_X(t)\) is the ratio between the \(X\)-norm and the
\(L^1\)-norm of the characteristic function of a set of measure
\(1/t\). Since \(\phi_X\) is quasiconcave and \(\phi_X(1)=1\),
\[
s\leq\phi_X(s)\leq1, \qquad 0<s\leq1.
\]
Moreover, \(\phi_X(s)/s\) is nonincreasing, and hence \(U_X\) is
nondecreasing. In particular,
\[
1\leq U_X(t)\leq t, \qquad t\geq1.
\]

For an Orlicz space \(L^\Phi(0,1)\), where \(\Phi\) is a Young
function normalized by \(\Phi(1)=1\), the fundamental function is
\[
\phi_{L^\Phi}(s) =\frac{1}{\Phi^{-1}(1/s)}.
\]
Consequently,
\[
U_{L^\Phi}(t) =\frac{t}{\Phi^{-1}(t)}.
\]

For \(x\in\Omega\), the lower local dimension of \(\mu\) at \(x\) is
defined by
\begin{equation*}
\underline d_\mu(x) = \liminf_{r\downarrow0} 
\frac{\log\mu(B(x,r))}{\log r}.
\end{equation*}

We use the standard notation \(\dim_H E\) for the Hausdorff
dimension of a set \(E\) (see \cite{HeinonenLectures,Falconer}). The
upper Hausdorff dimension of \(\mu\) is
\[
\dim_H^+\mu = \inf\{\dim_H F:\ F\subset\Omega\text{ Borel and }
\mu(\Omega\setminus F)=0\}.
\]
In particular,
\[
\dim_H^+\mu\leq\dim_H\Omega.
\]

We define the lower local order (see \cite[Definition~1.6]{Helfter})
by
\begin{equation*}
\underline{\operatorname{ord}}_{\mathrm{loc}}\mu(x) =
\liminf_{r\downarrow0} \frac{\log\log(1/\mu(B(x,r)))}{\log(1/r)}.
\end{equation*}


\subsection{Auxiliary lemmas}
\label{sec:pre-lemmas}

We first establish the functional characterization of the \(L\log
L\) threshold. We shall also use the relation between the upper
Hausdorff dimension of the measure and its lower local dimension.

For an Orlicz space \(L^\Phi(0,1)\), we have
\(U_\Phi(t)=t/\Phi^{-1}(t)\). Setting \(t=\Phi(s)\) gives
\(U_\Phi(\Phi(s))/\log\Phi(s) =(\Phi(s)/s)/\log\Phi(s)\). Since
\(\Phi(s)/s\) is nondecreasing, it follows that, for every
\(s_0>1\),
\begin{equation*}
 \sup_{t\geq2}\frac{U_\Phi(t)}{\log
t}=\infty \quad\Longleftrightarrow\quad \sup_{s\geq
s_0}\frac{\Phi(s)}{s\log s}=\infty.
\end{equation*}

For the Young function \(\Phi_0\) defined in (\ref{eq:Phi0}),
\begin{equation}
\label{eq:Phi0inverse} \Phi_0^{-1}(t)\simeq\frac{t}{\log(e+t)},
\qquad t\geq2.
\end{equation}
Consequently,
\begin{equation*}
U_{\Phi_0}(t)\simeq\log t, \qquad t\geq2.
\end{equation*}

The next lemma identifies this logarithmic growth precisely with the
embedding of \(L\log L\) into \(X\).

\begin{lemma}
\label{lem:LlogL-X} Let \(X\) be the r.i.\ space fixed above. The
following conditions are equivalent:
\begin{enumerate}
\item[\rm (i)]
\[
\sup_{t\geq2}\frac{U_X(t)}{\log t}<\infty.
\]

\item[\rm (ii)]
\[
L\log L(0,1)\hookrightarrow X.
\]
\end{enumerate}
Moreover, if these conditions hold, then for every measurable
\(E\subset\Omega\), with \(0<\mu(E)<\infty\),
\begin{equation}
\label{eq:LlogL-X-embedding} \|f\|_{X(E,\mu_E)} \leq
C_X\|f\|_{L^{\Phi_0}(E,\mu_E)}.
\end{equation}
\end{lemma}

\begin{proof}
Assume first that {\rm (i)} holds. Since \(U_X(t)=t\phi_X(1/t)\), we
obtain
\begin{equation*}
 \phi_X(s)\leq C_Xs\log(e/s), \qquad 0<s\leq1.
\end{equation*}

Set
\[
\varphi(s)=s\log(e/s), \qquad 0<s\leq1, \qquad \varphi(0)=0.
\]
Using the layer-cake representation and the triangle inequality,
first for simple functions and then by the Fatou property, for every
measurable \(h\) on \((0,1)\),
\[
\begin{aligned}
\|h\|_X &\leq \int_0^\infty
\phi_X\bigl(|\{|h|>\lambda\}|\bigr)\,d\lambda \\
&\leq C_X\int_0^\infty
\varphi\bigl(|\{|h|>\lambda\}|\bigr)\,d\lambda \\
&=C_X\int_0^1 h^*(s)\log(1/s)\,ds .
\end{aligned}
\]
The last expression is equivalent to the Luxemburg norm of
\(L^{\Phi_0}(0,1)\) (see \cite{BennettSharpley}), and hence
\[
L\log L(0,1)\hookrightarrow X.
\]

The same argument on the probability space \((E,\mu_E)\) gives, for
every measurable \(f\) on \(E\),
\[
\|f\|_{X(E,\mu_E)} \leq C_X\int_0^1 f_{\mu_E}^*(s)\log(1/s)\,ds
\preceq \|f\|_{L^{\Phi_0}(E,\mu_E)},
\]
which proves (\ref{eq:LlogL-X-embedding}).

Conversely, if \(L\log L(0,1)\hookrightarrow X\), then testing the
embedding on characteristic functions and using
(\ref{eq:Phi0inverse}) gives
\[
\phi_X(s)\preceq s\log(e/s), \qquad 0<s\leq\frac12.
\]
Therefore \(U_X(t)=t\phi_X(1/t)\preceq\log t\) for \(t\geq2\), which
proves {\rm (i)}.
\end{proof}

\begin{lemma}
\label{lem:dim} If \(\dim_H^+\mu<\infty\), then
\[
\underline d_\mu\in L^\infty(\Omega,\mu) \qquad\text{and}\qquad
\|\underline d_\mu\|_{L^\infty(\Omega,\mu)} \leq \dim_H^+\mu .
\]
\end{lemma}

\begin{proof}
Fix \(s>\dim_H^+\mu\), and choose a Borel set \(F\subset\Omega\)
such that
\[
\mu(\Omega\setminus F)=0 \qquad\text{and}\qquad \dim_HF<s.
\]
Then \(\mathcal H^s(F)=0\).

For integers \(m,k\geq1\), set
\[
A_{m,k} = \left\{ x\in F: \mu(B(x,r))\leq mr^s \text{ for every
}0<r<1/k \right\}.
\]
We claim that \(\mu^*(A_{m,k})=0\). Fix \(\eta>0\). Since \(\mathcal
H^s(F)=0\), there is a countable cover
\(A_{m,k}\subset\bigcup_iE_i\) such that
\[
\operatorname{diam}E_i<\frac1{4k}, \qquad
\sum_i(\operatorname{diam}E_i)^s<2^{-s}\eta.
\]
We may discard the sets which do not meet \(A_{m,k}\). Those of zero
diameter are singletons and hence form a countable set of
\(\mu\)-measure zero. For every remaining \(E_i\), choose \(x_i\in
A_{m,k}\cap E_i\) and
\[
\operatorname{diam}E_i<r_i<2\operatorname{diam}E_i.
\]
Then \(r_i<1/k\), the balls \(B(x_i,r_i)\) cover \(A_{m,k}\) up to a
countable set, and
\[
\sum_i r_i^s<\eta.
\]
Therefore
\[
\mu^*(A_{m,k}) \leq \sum_i\mu(B(x_i,r_i)) \leq m\sum_i r_i^s <m\eta.
\]
Letting \(\eta\downarrow0\) proves the claim.

Set
\[
N=(\Omega\setminus F)\cup\bigcup_{m,k\geq1}A_{m,k}.
\]
Then \(\mu^*(N)=0\). If \(x\notin N\), there are radii
\(r_j\downarrow0\) such that
\[
\mu(B(x,r_j))\geq r_j^s.
\]
Since \(\log r_j<0\),
\[
\frac{\log\mu(B(x,r_j))}{\log r_j}\leq s,
\]
and therefore
\[
\underline d_\mu(x)\leq s.
\]
Thus \(\underline d_\mu\leq s\) almost everywhere. Since this holds
for every \(s>\dim_H^+\mu\), we obtain
\[
\underline d_\mu(x)\leq\dim_H^+\mu \qquad\text{for almost every
}x\in\Omega.
\]
Hence
\[
\underline d_\mu\in L^\infty(\Omega,\mu) \qquad\text{and}\qquad
\|\underline d_\mu\|_{L^\infty(\Omega,\mu)} \leq\dim_H^+\mu.
\]
\end{proof}


\section{Preparatory results for the main theorems}
\label{sec:preparatory}

We establish here the functional and geometric estimates that will
be used in the proofs of the main theorems. We begin with the radial
estimates underlying the doubling arguments, then consider the
stronger small-scale growth associated with finite lower local
order, and finally establish the logarithmic endpoint for
log-concave measures.

\subsection{Radial estimates and local doubling}
\label{sec:main-estimates}

The results in this subsection prove the implication {\rm
(i)}\(\Rightarrow\){\rm (ii)} in Theorem~\ref{thm:finite-sharp}.
They require only the testing \((X,L^1)\)-Poincar\'e inequality
(\ref{eq:PI-B}), rather than the full \((X,L^1)\)-Poincar\'e
inequality (\ref{eq:intro-PI}).

Recall that
\[
V_x(r)=\mu(B(x,r)),\qquad D_x(r)=\frac{V_x(2r)}{V_x(r)}, \qquad
x\in\Omega,\ r>0.
\]

\begin{lemma}[Pointwise radial inequality]
\label{lem:radial-point} Let \(x\in\Omega\) and \(0<\rho<R'\). If
\(V_x\) is differentiable at \(\rho\) and
\[
V_x(R')\geq4V_x(\rho),
\]
then
\begin{equation}
\label{eq:radial-point} U_X\left(\frac{V_x(R')}{V_x(\rho)}\right)
\leq \frac32\,C_P R'\frac{V_x'(\rho)}{V_x(\rho)}.
\end{equation}
In particular, if \(V_x\) is differentiable at \(r\) and
\(D_x(r)\geq4\), then
\begin{equation}
\label{eq:radial-diff} U_X(D_x(r)) \leq 3C_P
r\frac{V_x'(r)}{V_x(r)}.
\end{equation}
\end{lemma}

\begin{proof}
For \(\varepsilon>0\), define the bounded Lipschitz radial cutoff
function centered at \(x\) by
\[
f_\varepsilon(y)=\eta_\varepsilon(d(x,y)),
\]
where
\[
\eta_\varepsilon(t)=
\begin{cases}
1, & 0\leq t\leq\rho,\\[1mm]
1-\dfrac{t-\rho}{\varepsilon},
& \rho<t<\rho+\varepsilon,\\[2mm]
0, & t\geq\rho+\varepsilon.
\end{cases}
\]
We claim that
\[
g_\varepsilon(y)=\frac1{\varepsilon}
\chi_{\{\rho<d(x,y)<\rho+\varepsilon\}}(y)
\]
is an upper gradient of \(f_\varepsilon\).

Indeed, let \(\gamma:[0,\ell]\to\Omega\) be a rectifiable curve
parametrized by arc length and set \(h(t)=d(x,\gamma(t))\). Then
\(h\) is \(1\)-Lipschitz, hence absolutely continuous, with
\(|h'(t)|\leq1\) almost everywhere. The function
\(f_\varepsilon\circ\gamma=\eta_\varepsilon\circ h\) is also
absolutely continuous. On \(\{\rho<h<\rho+\varepsilon\}\), the chain
rule gives
\[
|(f_\varepsilon\circ\gamma)'(t)| \leq\frac1{\varepsilon}
\]
almost everywhere, while on \(\{h<\rho\}\cup\{h>\rho+\varepsilon\}\)
the derivative vanishes almost everywhere.

At the two levels \(h=\rho\) and \(h=\rho+\varepsilon\),
\(\eta_\varepsilon\) is not differentiable, so the preceding chain
rule cannot be applied directly. However, since \(h\) is Lipschitz,
\(h'=0\) almost everywhere on each of these level sets (see
\cite[Theorem~3.3(i)]{EvansGariepy}). Hence, at every such point
where \(h'(t)=0\), the \(1/\varepsilon\)-Lipschitz property of
\(\eta_\varepsilon\) gives
\[
\limsup_{s\to0}
\frac{|\eta_\varepsilon(h(t+s))-\eta_\varepsilon(h(t))|}{|s|} \leq
\frac1{\varepsilon} \lim_{s\to0}\frac{|h(t+s)-h(t)|}{|s|} = 0.
\]
Thus
\[
|(f_\varepsilon\circ\gamma)'(t)| \leq g_\varepsilon(\gamma(t))
\]
for almost every \(t\in(0,\ell)\). By absolute continuity,
\[
|f_\varepsilon(\gamma(\ell))-f_\varepsilon(\gamma(0))| \leq
\int_\gamma g_\varepsilon\,ds.
\]
Thus \(g_\varepsilon\) is an upper gradient of \(f_\varepsilon\), as
claimed.

We now apply (\ref{eq:PI-B}) to the ball \(B=B(x,R')\), and write
\(V=V_x\) for simplicity. Since \(V\) is differentiable at \(\rho\)
and \(V(R')\geq4V(\rho)\), for all sufficiently small
\(\varepsilon>0\), we get
\[
\rho+\varepsilon<R', \qquad
\frac{V(\rho+\varepsilon)}{V(R')}\leq\frac13.
\]
Hence
\[
(f_\varepsilon)_{B,\mu}\leq\frac13,
\]
since \(0\leq f_\varepsilon\leq1\) and \(f_\varepsilon\) vanishes
outside \(B(x,\rho+\varepsilon)\); moreover, using that
\(f_\varepsilon=1\) on \(B(x,\rho)\), we obtain
\[
|f_\varepsilon-(f_\varepsilon)_{B,\mu}| \geq\frac23\chi_{B(x,\rho)}.
\]

Since, by hypothesis, \(Q=V(R')/V(\rho)\geq4\), and \(B=B(x,R')\),
we have
\[
\mu_B(B(x,\rho)) = \frac{\mu(B(x,\rho))}{\mu(B(x,R'))} =
\frac{V(\rho)}{V(R')} = \frac1Q.
\]
Therefore,
\[
\|f_\varepsilon-(f_\varepsilon)_{B,\mu}\|_{X(B,\mu_B)} \geq
\frac23\|\chi_{B(x,\rho)}\|_{X(B,\mu_B)} = \frac23\phi_X(1/Q).
\]

On the other hand, the support of \(g_\varepsilon\) is contained in
\(B\subset\sigma B\), and hence
\[
\int_{\sigma B}g_\varepsilon\,d\mu \leq
\frac{V(\rho+\varepsilon)-V(\rho)}{\varepsilon}.
\]
Applying the testing \((X,L^1)\)-Poincar\'e inequality
(\ref{eq:PI-B}), and recalling that \(r_B=R'\) and \(\mu(B)=V(R')\),
we obtain
\[
\frac23\phi_X(1/Q) \leq C_PR'\frac{V(\rho+\varepsilon)-V(\rho)}
{\varepsilon V(R')}.
\]
Letting \(\varepsilon\downarrow0\) and using the differentiability
of \(V\) at \(\rho\), we obtain
\[
\frac23\phi_X(1/Q) \leq C_PR'\frac{V'(\rho)}{V(R')}.
\]
Multiplying by \(3Q/2\) and using \(Q\phi_X(1/Q)=U_X(Q)\) and
\(Q/V(R')=1/V(\rho)\), we obtain (\ref{eq:radial-point}).

Finally, taking \(\rho=r\) and \(R'=2r\) gives
(\ref{eq:radial-diff}).
\end{proof}

\begin{lemma}[Radial growth inequality]
\label{lem:radial} Assume that the testing \((X,L^1)\)-Poincar\'e
inequality (\ref{eq:PI-B}) holds. Then there exists \(C_0\geq1\),
depending only on \(C_P\) and \(X\), such that
\begin{equation}
\label{eq:radial} U_X(D_x(2R)) \leq C_0\bigl(1+\log D_x(R)\bigr)
\end{equation}
for every \(x\in\Omega\) and every \(R>0\).
\end{lemma}

\begin{proof}
Fix \(x\in\Omega\) and write
\[
V(r)=V_x(r)=\mu(B(x,r)).
\]
The function \(V\) is positive, finite and nondecreasing, hence
differentiable almost everywhere. The function \(\log V\) is also
nondecreasing, and its distributional derivative is a positive Radon
measure whose absolutely continuous part has density \(V'/V\) (see
\cite[Theorem~3.28(c) and Corollary~3.29] {AmbrosioFuscoPallara}).
Consequently,
\[
\int_R^{2R}\frac{V'(r)}{V(r)}\,dr \leq \log V(2R)-\log V(R) = \log
D_x(R).
\]

Since \((R,2R)\) has length \(R\), there exists a point
\(\rho\in(R,2R)\) at which \(V\) is differentiable and
\begin{equation}
\label{eq:rho} R\frac{V'(\rho)}{V(\rho)} \leq \log D_x(R).
\end{equation}

If \(D_x(2R)<4\), then, by the monotonicity of \(U_X\) and the fact
that \(D_x(R)\geq1\),
\[
U_X(D_x(2R)) \leq U_X(4) \leq U_X(4)\bigl(1+\log D_x(R)\bigr).
\]
Hence we may assume that \(D_x(2R)\geq4\). Since \(\rho<2R\),
\[
Q=\frac{V(4R)}{V(\rho)} \geq \frac{V(4R)}{V(2R)} = D_x(2R) \geq4.
\]
Lemma~\ref{lem:radial-point}, applied with \(R'=4R\), gives
\[
U_X(Q) \leq 6C_PR\frac{V'(\rho)}{V(\rho)} \leq 6C_P\log D_x(R),
\]
where the last inequality follows from (\ref{eq:rho}). Since \(Q\geq
D_x(2R)\) and \(U_X\) is nondecreasing,
\[
U_X(D_x(2R)) \leq 6C_P\log D_x(R).
\]
Combining the two cases proves (\ref{eq:radial}) with, for instance,
\[
C_0=\max\{1,U_X(4),6C_P\}.
\]
\end{proof}

\begin{lemma}[Propagation of local doubling]
\label{lem:propagation} Assume that the testing
\((X,L^1)\)-Poincar\'e inequality (\ref{eq:PI-B}) holds and
\[
\sup_{t\geq2}\frac{U_X(t)}{\log t}=\infty.
\]
Then, for every \(D_0\geq1\), there exists \(M=M(D_0,C_P,X)<\infty\)
such that
\[
D_x(r_0)\leq D_0 \quad\Longrightarrow\quad D_x(r)\leq M
\qquad\text{for every }r\geq r_0.
\]
Moreover, if there exist \(D_0<\infty\) and a dense set
\(G\subset\Omega\) such that, for every \(x\in G\), there is a
sequence \(r_j(x)\downarrow0\) satisfying
\[
D_x(r_j(x))\leq D_0,\qquad j\geq1,
\]
then \(\mu\) is doubling, with doubling constant depending only on
\(D_0\), \(C_P\), and \(X\).
\end{lemma}

\begin{proof}
Let \(C_0\) be the constant in Lemma~\ref{lem:radial}. Fix
\(x\in\Omega\) and \(r_0>0\) with \(D_x(r_0)\leq D_0\), and set
\[
b_n=D_x(2^nr_0), \qquad n\geq0.
\]
Applying (\ref{eq:radial}) with \(R=2^nr_0\), we obtain
\begin{equation}
\label{eq:rec} U_X(b_{n+1})\leq C_0(1+\log b_n),\qquad n\geq0.
\end{equation}

Since \(U_X(t)\leq t\), the quotient \(U_X(t)/\log t\) is bounded on
compact subintervals of \([2,\infty)\). We may therefore choose
\(T\geq\max\{e,D_0\}\) such that
\begin{equation}
\label{eq:barrier} U_X(T)>2C_0\log T\geq C_0(1+\log T).
\end{equation}

We claim that \(b_n\leq T\) for every \(n\geq0\). Since \(b_0\leq
D_0\leq T\), suppose by contradiction that \(b_n>T\) for some \(n\),
and let \(n\geq1\) be the first such index. Then \(b_{n-1}\leq T\),
and the monotonicity of \(U_X\), together with (\ref{eq:barrier}),
gives
\[
U_X(b_n)\geq U_X(T)
>
C_0(1+\log T)\geq C_0(1+\log b_{n-1}),
\]
contradicting (\ref{eq:rec}). Hence
\begin{equation*}
 D_x(2^nr_0)\leq T,\qquad n\geq0.
\end{equation*}

Let \(r\geq r_0\), and choose \(n\geq0\) such that
\[
2^nr_0\leq r<2^{n+1}r_0.
\]
Then
\[
V_x(2r)\leq V_x(2^{n+2}r_0),\qquad V_x(r)\geq V_x(2^nr_0),
\]
and therefore
\[
\begin{aligned}
D_x(r) &\leq \frac{V_x(2^{n+2}r_0)}{V_x(2^nr_0)}
\\
&=D_x(2^{n+1}r_0)D_x(2^nr_0)\leq T^2.
\end{aligned}
\]
Thus the first conclusion holds with \(M=T^2\).

Suppose now that \(D_0\) and \(G\) satisfy the assumptions in the
second part. Fix \(x\in G\) and \(r>0\), and choose \(j\) such that
\(r_j(x)\leq r\). By the first part,
\begin{equation}
\label{eq:Gcontrol} D_x(r)\leq M,\qquad x\in G,\quad r>0.
\end{equation}

Fix \(x\in\Omega\) and \(r>0\). Since \(G\) is dense, choose \(y\in
G\) such that
\[
\delta=d(x,y)<r/4.
\]
Then
\[
B(x,2r)\subset B(y,2r+\delta),\qquad B(y,r-\delta)\subset B(x,r).
\]
Since \(2r+\delta<4(r-\delta)\),
\begin{equation}
\label{eq:center-inclusions} B(x,2r)\subset B(y,4(r-\delta)),\qquad
B(y,r-\delta)\subset B(x,r).
\end{equation}
Applying (\ref{eq:Gcontrol}) at \(y\), first with radius
\(r-\delta\) and then with radius \(2(r-\delta)\), gives
\[
\mu(B(y,4(r-\delta))) \leq M^2\mu(B(y,r-\delta)).
\]
Together with (\ref{eq:center-inclusions}), this yields
\[
\mu(B(x,2r))\leq M^2\mu(B(x,r)).
\]
Hence \(\mu\) is doubling.
\end{proof}

\begin{lemma}[Local dimension and doubling]
\label{lem:dim-point} Assume that the testing \((X,L^1)\)-Poincar\'e
inequality (\ref{eq:PI-B}) holds and
\[
\sup_{t\geq2}\frac{U_X(t)}{\log t}=\infty.
\]

\begin{enumerate}
\item[\rm (a)]
For every \(x\in\Omega\), the following conditions are equivalent:
\begin{enumerate}
\item[\rm (i)] \(\sup_{r>0}D_x(r)<\infty\);
\item[\rm (ii)] \(\liminf_{r\downarrow0}D_x(r)<\infty\);
\item[\rm (iii)] \(\underline d_\mu(x)<\infty\).
\end{enumerate}
Moreover, if \(s>\underline d_\mu(x)\), then
\[
\sup_{r>0}D_x(r)\leq M(2^{s+1},C_P,X),
\]
where \(M\) is given by Lemma~\ref{lem:propagation}; if \(D_x(r)\leq
C\) for every \(r>0\), then
\[
\underline d_\mu(x)\leq\log_2 C.
\]

\item[\rm (b)]
The measure \(\mu\) is doubling if and only if
\[
\underline d_\mu\in L^\infty(\Omega,\mu).
\]
Moreover, if
\[
\|\underline d_\mu\|_{L^\infty(\Omega,\mu)}\leq s_0,
\]
then the doubling constant depends only on \(s_0\), \(C_P\), and
\(X\).
\end{enumerate}
\end{lemma}

\begin{proof}
We first show that, if \(s>\underline d_\mu(x)\), then
\[
\liminf_{r\downarrow0}D_x(r)\leq2^s.
\]
Suppose otherwise. Then there exists \(R_x>0\) such that
\[
\mu(B(x,r))<2^{-s}\mu(B(x,2r)),\qquad 0<r<R_x.
\]
Iteration at dyadic radii gives
\[
\mu(B(x,2^{-j}R_x)) \leq2^{-js}\mu(B(x,R_x)),\qquad j\geq0.
\]
If
\[
2^{-(j+1)}R_x<r\leq2^{-j}R_x,
\]
then
\[
\mu(B(x,r)) \leq 2^s\frac{\mu(B(x,R_x))}{R_x^s}\,r^s.
\]
Taking logarithms and dividing by \(\log r<0\), then letting
\(r\downarrow0\), gives
\[
\underline d_\mu(x)\geq s,
\]
contradicting the choice of \(s\).

This proves {\rm (iii)}\(\Rightarrow\){\rm (ii)}. Clearly, {\rm (i)}
implies {\rm (ii)}. Assume {\rm (ii)}, and choose
\[
D_0>\liminf_{r\downarrow0}D_x(r).
\]
There are radii \(r_j\downarrow0\) such that \(D_x(r_j)\leq D_0\).
By Lemma~\ref{lem:propagation},
\[
D_x(r)\leq M(D_0,C_P,X),\qquad r\geq r_j.
\]
Letting \(j\to\infty\) proves {\rm (i)}. If \(s>\underline
d_\mu(x)\), the first estimate above allows us to take
\(D_0=2^{s+1}\), which gives the stated bound.

To complete part {\rm (a)}, suppose that \(D_x(r)\leq C\) for every
\(r>0\). Then
\[
\mu(B(x,2^{-k})) \geq C^{-k}\mu(B(x,1)),\qquad k\geq0.
\]
Consequently,
\[
\frac{\log\mu(B(x,2^{-k}))}{\log 2^{-k}} \leq \log_2
C-\frac{\log\mu(B(x,1))}{k\log2},
\]
and hence
\[
\underline d_\mu(x)\leq\log_2 C.
\]
This proves {\rm (i)}\(\Rightarrow\){\rm (iii)} and completes the
proof of part {\rm (a)}.

We now prove part {\rm (b)}. If \(\mu\) is doubling with constant
\(C_D\), part {\rm (a)} gives
\[
\underline d_\mu(x)\leq\log_2 C_D,\qquad x\in\Omega,
\]
and hence \(\underline d_\mu\in L^\infty(\Omega,\mu)\).

Conversely, assume that
\[
\|\underline d_\mu\|_{L^\infty(\Omega,\mu)}\leq s_0,
\]
and let
\[
G=\{x\in\Omega:\underline d_\mu(x)\leq s_0\}.
\]
Then \(G\) has full measure and is therefore dense in \(\Omega\),
since every ball has positive measure. For \(x\in G\), the first
estimate in part {\rm (a)}, applied with \(s=s_0+1/2\), gives
\[
\liminf_{r\downarrow0}D_x(r)\leq2^{s_0+1/2}.
\]
Hence there exists a sequence \(r_j(x)\downarrow0\) such that
\[
D_x(r_j(x))\leq2^{s_0+1}.
\]
Lemma~\ref{lem:propagation}, applied with \(D_0=2^{s_0+1}\), shows
that \(\mu\) is doubling, with doubling constant depending only on
\(s_0\), \(C_P\), and \(X\).
\end{proof}

\subsection{Lower local order and small-ball growth}

We next consider the stronger functional condition
\[
\frac{U_X(t)}{\log t}\longrightarrow\infty \qquad\text{as
}t\to\infty.
\]
Recall that the lower local order is defined by
\[
\underline{\operatorname{ord}}_{\rm loc}\mu(x) =
\liminf_{r\downarrow0} \frac{\log\log(1/\mu(B(x,r)))}{\log(1/r)}.
\]
Under the stronger functional condition above, the radial growth
inequality can be iterated towards smaller scales. We first relate
finite lower local order to a small-ball condition.

\begin{lemma}
\label{lem:order-SB} For \(x\in\Omega\), the following conditions
are equivalent:
\begin{enumerate}
\item[\rm (i)]
\[
\underline{\operatorname{ord}}_{\rm loc}\mu(x)<\infty.
\]

\item[\rm (ii)]
There exist \(C,\beta>0\) and a sequence \(r_j\downarrow0\) such
that
\begin{equation}
\label{eq:SB} \mu(B(x,r_j)) \geq \exp(-Cr_j^{-\beta}), \qquad
j\geq1.
\end{equation}
\end{enumerate}
Moreover, if \(\underline d_\mu(x)<\infty\), then
\[
\underline{\operatorname{ord}}_{\rm loc}\mu(x)=0.
\]
\end{lemma}

\begin{proof}
Suppose first that
\[
L=\underline{\operatorname{ord}}_{\rm loc}\mu(x)<\infty.
\]
Choose \(\beta>L\). By the definition of the lower local order,
there is a sequence \(r_j\downarrow0\) such that
\[
\frac{\log\log(1/\mu(B(x,r_j)))}{\log(1/r_j)} \leq\beta
\]
for all sufficiently large \(j\). Hence
\[
\log\frac1{\mu(B(x,r_j))} \leq r_j^{-\beta},
\]
which gives (\ref{eq:SB}), after discarding finitely many terms.

Conversely, if (\ref{eq:SB}) holds, then
\[
\log\frac1{\mu(B(x,r_j))} \leq Cr_j^{-\beta},
\]
and therefore
\[
\frac{\log\log(1/\mu(B(x,r_j)))}{\log(1/r_j)} \leq \beta+\frac{\log
C}{\log(1/r_j)}.
\]
Taking the lower limit gives
\[
\underline{\operatorname{ord}}_{\rm loc}\mu(x) \leq\beta<\infty.
\]

Finally, suppose that \(\underline d_\mu(x)<\infty\), and choose
\(s>\underline d_\mu(x)\). By the definition of the lower local
dimension, there is a sequence \(r_j\downarrow0\) such that
\[
\mu(B(x,r_j))\geq r_j^s.
\]
Thus
\[
0\leq \frac{\log\log(1/\mu(B(x,r_j)))}{\log(1/r_j)} \leq \frac{\log
s+\log\log(1/r_j)}{\log(1/r_j)} \longrightarrow0,
\]
and hence
\[
\underline{\operatorname{ord}}_{\rm loc}\mu(x)=0.
\]
\end{proof}

\begin{lemma}[Small-ball growth and doubling]
\label{lem:vertical} Assume that the testing \((X,L^1)\)-Poincar\'e
inequality (\ref{eq:PI-B}) holds and
\[
\frac{U_X(t)}{\log t}\longrightarrow\infty \qquad\text{as
}t\to\infty.
\]
Then there exists \(T_0\geq1\), depending only on \(C_P\) and \(X\),
with the following property. If \(x\in\Omega\), \(R>0\), and
\begin{equation}
\label{eq:vertical-hyp} \log D_x(2R)\geq T_0,
\end{equation}
then, for every \(\beta>0\),
\begin{equation}
\label{eq:superpower} \lim_{s\downarrow0}
s^\beta\log\frac1{\mu(B(x,s))} = \infty.
\end{equation}
Consequently, if
\[
\underline{\operatorname{ord}}_{\rm loc}\mu(x)<\infty,
\]
then
\begin{equation}
\label{eq:order-doubling} \sup_{r>0}D_x(r)\leq e^{T_0}.
\end{equation}
\end{lemma}

\begin{proof}
Let \(C_0\) be the constant of Lemma~\ref{lem:radial}. For every
\(\Lambda\geq1\), there exists \(T(\Lambda)\geq1\), depending only
on \(C_P\), \(X\), and \(\Lambda\), such that
\[
U_X(t) \geq C_0(\Lambda+1)\log t \qquad\text{whenever }\log t\geq
T(\Lambda).
\]

Fix \(x\), write
\[
\ell(r)=\log D_x(r),
\]
and set \(T_0=T(3)\). We first claim that
\begin{equation}
\label{eq:vertical-step} \ell(r)\geq T(\Lambda)
\quad\Longrightarrow\quad \ell(r/2)\geq\Lambda\ell(r).
\end{equation}
Indeed, applying (\ref{eq:radial}) with \(R=r/2\) gives
\[
C_0(\Lambda+1)\ell(r) \leq U_X(D_x(r)) \leq
C_0\bigl(1+\ell(r/2)\bigr).
\]
Since \(\ell(r)\geq T(\Lambda)\geq1\),
\[
\ell(r/2) \geq (\Lambda+1)\ell(r)-1 \geq \Lambda\ell(r),
\]
which proves (\ref{eq:vertical-step}).

Assume now that (\ref{eq:vertical-hyp}) holds, and set
\[
r_k=2^{1-k}R,\qquad k\geq0.
\]
Applying (\ref{eq:vertical-step}) with \(\Lambda=3\) repeatedly, we
obtain
\[
\ell(r_k)\geq3^k\ell(r_0), \qquad k\geq0.
\]
In particular, \(\ell(r_k)\to\infty\).

Given \(\Lambda\geq3\), choose \(k_0\) such that
\[
\ell(r_{k_0})\geq T(\Lambda).
\]
A further iteration of (\ref{eq:vertical-step}) gives
\[
\ell(r_k)\geq\Lambda^{k-k_0}, \qquad k\geq k_0.
\]

Set
\[
H(r)=\log\frac1{V_x(r)}.
\]
Since
\[
H(r)-H(2r)=\ell(r),
\]
telescoping gives
\[
H(r_k) = H(4R)+\sum_{j=0}^k\ell(r_j) \geq H(4R)+\Lambda^{k-k_0}.
\]
If \(r_{k+1}<s\leq r_k\), then, by monotonicity of \(V_x\),
\(H(s)\geq H(r_k)\), while \(r_{k+1}=2^{-k}R<s\) gives \(2^k>R/s\).
Hence, for all sufficiently small \(s\),
\[
\log\frac1{\mu(B(x,s))} \geq H(4R) + \Lambda^{-k_0}
\left(\frac{R}{s}\right)^{\log_2\Lambda}.
\]
Choosing \(\Lambda>\max\{3,2^\beta\}\), we get
(\ref{eq:superpower}).

Finally, suppose that
\[
\underline{\operatorname{ord}}_{\rm loc}\mu(x)<\infty.
\]
By Lemma~\ref{lem:order-SB}, there exist \(C,\beta>0\) and a
sequence \(r_j\downarrow0\) such that
\[
\mu(B(x,r_j)) \geq \exp(-Cr_j^{-\beta}).
\]
To prove (\ref{eq:order-doubling}), assume that
\[
\log D_x(r)\geq T_0
\]
for some \(r>0\). Applying the first part of the lemma with
\(2R=r\), we obtain
\[
r_j^\beta\log\frac1{\mu(B(x,r_j))} \longrightarrow\infty.
\]
On the other hand,
\[
r_j^\beta\log\frac1{\mu(B(x,r_j))} \leq C,
\]
which is impossible. Therefore
\[
\log D_x(r)<T_0 \qquad\text{for every }r>0,
\]
and hence
\[
\sup_{r>0}D_x(r)\leq e^{T_0},
\]
as required.
\end{proof}

\subsection{The logarithmic endpoint}
\label{sec:log-endpoint}

An absolutely continuous log-concave measure on \(\mathbb R^n\) has
density \(e^{-V}\), where \(V:\mathbb R^n\to(-\infty,\infty]\) is
convex. Restrictions to convex sets remain log-concave. For a
probability measure \(\nu\), write
\[
\nu^+(A)=\liminf_{\varepsilon\downarrow0}
\frac{\nu(A_\varepsilon)-\nu(A)}{\varepsilon}, \qquad
A_\varepsilon=\{x:\operatorname{dist}(x,A)<\varepsilon\}.
\]
Its isoperimetric profile is
\[
I_\nu(s)=\inf\{\nu^+(A):\nu(A)=s\}, \qquad 0<s<1,
\]
and we set \(I_\nu(0)=0\) (see \cite{BobkovHoudre}).

Bobkov's dimension-free isoperimetric inequality for log-concave
probability measures (see \cite[Corollary~2.2]{Bobkov}) yields the
following local estimate and its \(L\log L\)-Poincar\'e consequence.

\begin{proposition}[Dimension-free logarithmic endpoint]
\label{prop:logPI} Let \(\mu\) be an absolutely continuous
log-concave measure on \(\mathbb R^n\), and let \(B=B(x_0,r_B)\) be
a Euclidean ball such that \(\mu(B)>0\). Then
\begin{equation}
\label{eq:logiso} 2r_B I_{\mu_B}(s) \geq s\log\frac1s +
(1-s)\log\frac1{1-s}, \qquad 0<s\leq\frac12.
\end{equation}
In particular,
\begin{equation}
\label{eq:logisophi} I_{\mu_B}(s) \succeq \frac1{r_B}s\log\frac es,
\qquad 0<s\leq\frac12.
\end{equation}
Moreover,
\begin{equation}
\label{eq:logPI} \|f-f_{B,\mu}\|_{L^{\Phi_0}(B,\mu_B)} \preceq
r_B\int_Bg\,d\mu_B
\end{equation}
for every measurable \(f\in L^1(B,\mu_B)\) and every \(g\in
L^1(B,\mu_B)\) which is an upper gradient of \(f\) along rectifiable
curves contained in \(B\). The implicit constant is absolute.
\end{proposition}

\begin{proof}
Since \(B\) is convex, the normalized restriction \(\mu_B\) is a
log-concave probability measure concentrated in \(\overline
B(x_0,r_B)\). Hence (\ref{eq:logiso}) follows directly from
\cite[Corollary~2.2]{Bobkov}.

We next prove the \(L\log L\)-Poincar\'e estimate (\ref{eq:logPI}).
For simplicity, write
\[
\rho=\mu_B,\qquad \varphi(s)=s\log(e/s),\quad 0\leq s\leq1.
\]
First assume that \(f\) is Lipschitz. Let \(m_f\) be a median of
\(f\),\footnote{That is, \(\rho(\{f>m_f\})\leq1/2\) and
\(\rho(\{f<m_f\})\leq1/2\).} and set
\[
h_+=(f-m_f)_+,\qquad h_-=(m_f-f)_+.
\]
For either choice \(h=h_+\) or \(h=h_-\), the distribution function
\(\lambda_h(t)=\rho\{h>t\}\) satisfies \(\lambda_h(t)\leq1/2\) for
\(t>0\).

By the coarea form of the isoperimetric--Sobolev inequality (see
\cite[Sections~2--4]{BobkovHoudre}), (\ref{eq:logisophi}), and
(\ref{eq:Phi0inverse}), we obtain
\begin{equation}
\label{eq:coarea-log}
\begin{aligned}
r_B\int_B|\nabla h|\,d\rho &\geq r_B\int_0^\infty
I_\rho(\lambda_h(t))\,dt \succeq
\int_0^\infty\varphi(\lambda_h(t))\,dt
\\
&\succeq \int_0^\infty \phi_{L^{\Phi_0}}(\lambda_h(t))\,dt.
\end{aligned}
\end{equation}
By the layer-cake representation and the triangle inequality,
\[
\|h\|_{L^{\Phi_0}(B,\rho)} \leq \int_0^\infty
\|\chi_{\{h>t\}}\|_{L^{\Phi_0}(B,\rho)}\,dt = \int_0^\infty
\phi_{L^{\Phi_0}}(\lambda_h(t))\,dt.
\]
Hence, by (\ref{eq:coarea-log}),
\[
\|h\|_{L^{\Phi_0}(B,\rho)} \preceq r_B\int_B|\nabla h|\,d\rho.
\]
Applying the preceding estimate to \(h_+\) and \(h_-\), and using
\(|\nabla h_+|+|\nabla h_-|=|\nabla f|\) almost everywhere, we
obtain
\[
\|f-m_f\|_{L^{\Phi_0}(B,\rho)} \preceq r_B\int_B|\nabla f|\,d\rho.
\]
Since \(\Phi_0(1)=1\), we have
\(\|u\|_{L^1(B,\rho)}\leq\|u\|_{L^{\Phi_0}(B,\rho)}\) and
\(\|1\|_{L^{\Phi_0}(B,\rho)}=1\). Hence
\[
\|f-f_{B,\mu}\|_{L^{\Phi_0}(B,\rho)} \leq
2\|f-m_f\|_{L^{\Phi_0}(B,\rho)} \preceq r_B\int_B|\nabla f|\,d\rho.
\]

We now pass to a general \(f\in L^1(B,\rho)\) with an integrable
upper gradient \(g\). Let
\[
D=B\cap\operatorname{int}(\operatorname{supp}\mu).
\]
Since \(\mu\) is absolutely continuous and log-concave, \(D\) is
open and convex, with \(\rho(D)=1\). Moreover, the density \(w\) of
\(\rho\) has a representative which is positive and continuous on
\(D\), and hence is bounded above and bounded away from zero on
compact subsets of \(D\). Consequently,
\[
f,g\in L^1_{\rm loc}(D,dx).
\]

Fix a unit vector \(v\). By Fubini's theorem, \(g\) is locally
integrable on almost every line parallel to \(v\). The
upper-gradient inequality on subsegments of such lines shows that
\(f\) is locally absolutely continuous there and
\[
\left|\frac{d}{dt}f(x+tv)\right| \leq g(x+tv)
\]
for almost every \(t\).

Taking \(v\) successively to be the coordinate directions, the
standard characterization of Sobolev functions by absolute
continuity on lines gives
\[
f\in W^{1,1}_{\rm loc}(D)
\]
(see \cite[Chapter~4]{EvansGariepy}). Repeating the preceding
argument for every \(v\) in a countable dense subset \(\mathcal D\)
of the Euclidean unit sphere gives
\[
|\nabla f(x)\cdot v|\leq g(x), \qquad v\in\mathcal D,
\]
outside a null set depending on \(v\). Since \(\mathcal D\) is
countable, the union of these exceptional null sets is still null.
Fix \(x\) outside this union. Then
\[
|\nabla f(x)\cdot v|\leq g(x), \qquad v\in\mathcal D.
\]
For any \(u\in S^{n-1}\), choose \(v_k\in\mathcal D\) with \(v_k\to
u\). By continuity of \(v\mapsto\nabla f(x)\cdot v\),
\[
|\nabla f(x)\cdot u| = \lim_{k\to\infty}|\nabla f(x)\cdot v_k| \leq
g(x).
\]
Hence the same inequality holds for every unit vector \(u\), and
therefore
\[
|\nabla f(x)| = \sup_{|u|=1}|\nabla f(x)\cdot u| =
\sup_{v\in\mathcal D}|\nabla f(x)\cdot v| \leq g(x).
\]
Thus
\[
|\nabla f|\leq g \qquad\text{a.e. in }D.
\]

Choose \(x_*\in D\) and \(0<t_j\uparrow1\), and set
\[
K_j=x_*+t_j(\overline D-x_*),\qquad
\rho_j=\frac{\rho|_{K_j}}{\rho(K_j)}.
\]
Then \(K_j\Subset D\), \(K_j\uparrow D\), and each \(\rho_j\) is
log-concave and concentrated in \(\overline B(x_0,r_B)\).

Fix \(j\). Since \(f\in W^{1,1}_{\rm loc}(D)\), standard local
approximation by mollification gives smooth functions \(f_\ell\) on
a neighborhood of \(K_j\) such that
\[
f_\ell\to f,\qquad \nabla f_\ell\to\nabla f
\]
in \(L^1(K_j,\rho_j)\) (see \cite[Chapter~4]{EvansGariepy}). Since
\(K_j\) is compact and convex, \(f_\ell|_{K_j}\) is Lipschitz and
admits a Lipschitz extension to \(\mathbb R^n\). Applying the
Lipschitz estimate proved above to \(\rho_j\), we obtain
\[
\|f_\ell-(f_\ell)_{\rho_j}\|_{L^{\Phi_0}(K_j,\rho_j)} \leq
Cr_B\int_{K_j}|\nabla f_\ell|\,d\rho_j,
\]
where \(u_{\rho_j}=\int u\,d\rho_j\). Passing to an almost
everywhere convergent subsequence and using the Fatou property gives
\[
\|f-f_{\rho_j}\|_{L^{\Phi_0}(K_j,\rho_j)} \leq Cr_B\int_{K_j}|\nabla
f|\,d\rho_j \leq Cr_B\int_{K_j}g\,d\rho_j.
\]

We now let \(j\to\infty\). Since \(\rho(D)=1\), we have
\[
\rho(K_j)\to1,\qquad f_{\rho_j}\to f_{B,\mu},\qquad
\int_{K_j}g\,d\rho_j\to\int_Bg\,d\rho.
\]
Thus, for every \(\lambda>Cr_B\int_Bg\,d\rho\), the preceding
estimate gives, for all sufficiently large \(j\),
\[
\int_{K_j} \Phi_0\left(\frac{|f-f_{\rho_j}|}{\lambda}\right)\,d\rho
\leq\rho(K_j).
\]
Fatou's lemma yields
\[
\int_B \Phi_0\left(\frac{|f-f_{B,\mu}|}{\lambda}\right)\,d\rho
\leq1.
\]
Taking the infimum over all such \(\lambda\) proves
(\ref{eq:logPI}).
\end{proof}

Since the standard Gaussian measures are non-doubling,
Proposition~\ref{prop:logPI} immediately gives the logarithmic
endpoint examples needed below.

\begin{corollary}
\label{cor:Gaussian} For every \(n\geq1\), let \(\gamma_n\) be the
standard Gaussian measure on \(\mathbb R^n\). Then \(\gamma_n\) is
non-doubling, but satisfies
\begin{equation*}
\|f-f_{B,\gamma_n}\|_{L^{\Phi_0}(B,(\gamma_n)_B)} \preceq
r_B\frac1{\gamma_n(B)} \int_Bg\,d\gamma_n
\end{equation*}
for every Euclidean ball \(B\), every measurable \(f\in
L^1(B,\gamma_n)\), and every integrable upper gradient \(g\) of
\(f\) along curves in \(B\). The implicit constant is independent of
\(n\), \(B\), \(f\), and \(g\).
\end{corollary}

\section{Proofs of the main theorems}
\label{sec:main-proofs}

We now combine the results established in the preceding section to
prove the two main theorems.

\subsection{Proof of Theorem~\ref{thm:finite-sharp}}
\label{sec:finite-proof}

\begin{proof}
By Lemma~\ref{lem:LlogL-X}, conditions {\rm (i)} and {\rm (iii)} are
equivalent. To prove {\rm (i)}\(\Rightarrow\){\rm (ii)}, let
\((\Omega,d,\mu)\) satisfy the assumptions in {\rm (ii)} and assume
that {\rm (i)} holds. Then, by Lemma~\ref{lem:LlogL-X},
\[
\sup_{t\geq2}\frac{U_X(t)}{\log t}=\infty.
\]
Since \(\underline d_\mu\in L^\infty(\Omega,\mu)\), part {\rm (b)}
of Lemma~\ref{lem:dim-point} implies that \(\mu\) is doubling, with
the stated dependence of the constant.

To prove {\rm (ii)}\(\Rightarrow\){\rm (i)}, suppose that {\rm (i)}
fails. By Lemma~\ref{lem:LlogL-X},
\[
L^{\Phi_0}(0,1)\hookrightarrow X.
\]
Combining this embedding with Corollary~\ref{cor:Gaussian}, the
standard Gaussian measure \(\gamma_1\) satisfies the testing
\((X,L^1)\)-Poincar\'e inequality (\ref{eq:PI-B}) with \(\sigma=1\).
Moreover, \(\gamma_1\) has a continuous positive density, and hence
\[
\underline d_{\gamma_1}\equiv1,
\]
whereas \(\gamma_1\) is not doubling. Thus {\rm (ii)} fails, and
therefore {\rm (ii)}\(\Rightarrow\){\rm (i)}.

It remains to verify the fixed-dimensional assertion. Fix
\(n\geq1\). Under {\rm (i)}, every measure \(\mu\) as in the
statement satisfies
\[
\dim_H^+\mu\leq n.
\]
By Lemma~\ref{lem:dim},
\[
\underline d_\mu\in L^\infty(\mathbb R^n,\mu),
\]
and hence part {\rm (b)} of Lemma~\ref{lem:dim-point} shows that
\(\mu\) is doubling. If {\rm (i)} fails, the same embedding and
Corollary~\ref{cor:Gaussian} show that \(\gamma_n\) satisfies
(\ref{eq:PI-B}) but is not doubling. This proves the
fixed-dimensional characterization.
\end{proof}

Since
\[
\dim_H^+\mu\leq\dim_H\Omega,
\]
Lemma~\ref{lem:dim} and part {\rm (b)} of Lemma~\ref{lem:dim-point}
immediately give the following consequence.

\begin{corollary}
\label{cor:finite} Assume that the testing \((X,L^1)\)-Poincar\'e
inequality (\ref{eq:PI-B}) holds and
\[
\sup_{t\geq2}\frac{U_X(t)}{\log t}=\infty.
\]
If \(\dim_H\Omega<\infty\), then \(\mu\) is doubling.
\end{corollary}

\begin{remark}
\label{rem:1d} In dimension one, the Poincar\'e inequality can be
written in terms of the derivative. Let \(J\subset\mathbb R\) be a
nondegenerate interval, and let \(d\mu=w\,dx\), with \(w>0\) almost
everywhere and \(\mu(I)<\infty\) for every bounded interval
\(I\subset J\). Consider
\begin{equation}
\label{eq:PI1d} \|f-f_{I,\mu}\|_{X(I,\mu_I)} \leq
C_P|I|\,\|f'\|_{L^1(I,\mu_I)}
\end{equation}
for every bounded nondegenerate interval \(I\subset J\) and every
Lipschitz function \(f\). Every ball of \(J\) is an interval \(I\)
with \(|I|\leq2r_I\), and the radial cutoffs in the proof of
Lemma~\ref{lem:radial-point} satisfy
\(|f_\varepsilon'|=g_\varepsilon\) almost everywhere. Hence the
arguments above apply to (\ref{eq:PI1d}), up to an absolute change
in the constants.

Since \(\underline d_\mu(x)=1\) at every Lebesgue point \(x\) of
\(w\) with \(0<w(x)<\infty\), part {\rm (b)} of
Lemma~\ref{lem:dim-point} shows that
\[
\sup_{t\geq2}\frac{U_X(t)}{\log t}=\infty
\]
implies that every such measure satisfying (\ref{eq:PI1d}) is
doubling on \(J\).

Conversely, suppose that this supremum is finite. By
Lemma~\ref{lem:LlogL-X}, \(L\log L(0,1)\hookrightarrow X\). For
\(J=\mathbb R\), Corollary~\ref{cor:Gaussian} provides a
non-doubling measure satisfying (\ref{eq:PI1d}). If \(J\) has a
finite left endpoint \(a\), take
\[
w(x)=\exp\left(-\frac1{x-a}\right), \qquad x\in\operatorname{int}J.
\]
This density is log-concave, and the one-dimensional Lipschitz
estimate in Lemma~\ref{lem:normed-logPI}, applied to each normalized
restriction \(\mu_I\), gives (\ref{eq:PI1d}). On the other hand,
\[
\mu((a,a+t))\sim t^2e^{-1/t}, \qquad t\downarrow0,
\]
and therefore
\[
\frac{\mu(B_J(a+r,2r))}{\mu(B_J(a+r,r))} =
\frac{\mu((a,a+3r))}{\mu((a,a+2r))} \longrightarrow\infty.
\]
Thus \(\mu\) is not doubling. The case of a finite right endpoint
follows by reflection. Consequently, for every fixed nondegenerate
interval \(J\), the condition above is also necessary for all
measures satisfying (\ref{eq:PI1d}) to be doubling.

For characterizations of weighted one-dimensional Orlicz--Poincar\'e
inequalities in terms of the weight (see
\cite{KorobenkoMilshsteinYong}); here we ask instead when their
uniform local validity forces the measure to be doubling.
\end{remark}

\subsection{Proof of Theorem~\ref{thm:SB}}
\label{sec:SB-proof}

\begin{proof}
Since \(U_X(t)/\log t\to\infty\), we have \(\sup_{t\geq2}U_X(t)/\log
t=\infty\). Hence part {\rm (b)} of Lemma~\ref{lem:dim-point} shows
that {\rm (i)} implies {\rm (ii)}.

Assume now {\rm (ii)}. Then \(\underline d_\mu(x)<\infty\) for
almost every \(x\in\Omega\), and Lemma~\ref{lem:order-SB} gives
\[
\underline{\operatorname{ord}}_{\rm loc}\mu(x)=0 \qquad\text{for
almost every }x\in\Omega.
\]
Thus {\rm (iii)} holds.

It remains to prove {\rm (iii)}\(\Rightarrow\){\rm (i)}. Let
\[
G= \left\{ x\in\Omega: \underline{\operatorname{ord}}_{\rm
loc}\mu(x)<\infty \right\}.
\]
By {\rm (iii)}, \(G\) has full measure and is therefore dense in
\(\Omega\), since every ball has positive measure. By
Lemma~\ref{lem:vertical},
\[
D_x(r)\leq e^{T_0}, \qquad x\in G,\quad r>0,
\]
where \(T_0\) depends only on \(C_P\) and \(X\). Therefore the
second part of Lemma~\ref{lem:propagation}, applied with
\(D_0=e^{T_0}\), shows that \(\mu\) is doubling. Since \(T_0\)
depends only on \(C_P\) and \(X\), so does the doubling constant.
This proves {\rm (iii)}\(\Rightarrow\){\rm (i)} and completes the
proof.
\end{proof}

\begin{remark}
In fact, the implication to doubling remains valid if
\(\underline{\operatorname{ord}}_{\rm loc}\mu(x)<\infty\) merely on
a dense subset of \(\Omega\), without any uniform bound.
\end{remark}

\section{Sharp logarithmic endpoints in infinite-dimensional spaces}
\label{sec:infinite}

We now turn to infinite-dimensional spaces. Under our standing
assumptions, a doubling measure with full support forces the
underlying metric space to have finite Hausdorff dimension (see
\cite{HeinonenLectures}). Since the models considered below have
infinite Hausdorff dimension, their measures cannot be doubling.
Nevertheless, the \(L\log L\)-Poincar\'e endpoint remains valid and
is sharp. More precisely, for the product, Gaussian Hilbert, and
Wiener models considered below, the \((X,L^1)\)-Poincar\'e
inequality holds if and only if $L\log L(0,1)\hookrightarrow X$.

\subsection{Testing criteria}
\label{sec:criteria}

The next lemma captures the mechanism behind the logarithmic
obstruction: the oscillation remains bounded below on sets of small
relative measure, while the integral of an upper gradient has the
corresponding logarithmic size.

\begin{lemma}
\label{lem:test} Assume that \((\Omega,d,\mu)\) supports an
\((X,L^1)\)-Poincar\'e inequality (\ref{eq:intro-PI}). Suppose that
there exist constants \(a,A>0\) and \(\Lambda\geq1\), and, for each
\(n\geq1\), a ball \(B_n\), a measurable set \(E_n\subset B_n\), a
Lipschitz function \(f_n:\Omega\to\mathbb R\), and an upper gradient
\(g_n\in L^1(\sigma B_n,\mu)\) of \(f_n\). Set
\[
p_n=\frac{\mu(E_n)}{\mu(B_n)}.
\]
Assume that
\begin{equation}
\label{eq:testp} 0<p_{n+1}<p_n\leq\frac12,\qquad
p_n\longrightarrow0,\qquad \log\frac1{p_{n+1}} \leq
\Lambda\log\frac1{p_n},
\end{equation}
and
\begin{equation}
\label{eq:testosc} |f_n-(f_n)_{B_n,\mu}| \geq a \qquad\text{on }E_n,
\end{equation}
while
\begin{equation*}
 r_{B_n}\frac1{\mu(\sigma B_n)} \int_{\sigma
B_n}g_n\,d\mu \leq A p_n\log\frac1{p_n}.
\end{equation*}
Then
\begin{equation}
\label{eq:testU} U_X(t)\preceq\log t, \qquad t\geq2.
\end{equation}
In particular,
\[
\sup_{t\geq2}\frac{U_X(t)}{\log t}<\infty.
\]
\end{lemma}

\begin{proof}
The lattice property, (\ref{eq:testosc}), and (\ref{eq:intro-PI})
give
\[
a\phi_X(p_n) \leq \|f_n-(f_n)_{B_n,\mu}\|_{X(B_n,\mu_{B_n})} \leq
C_PA p_n\log\frac1{p_n}.
\]
Consequently,
\[
U_X(1/p_n) = \frac{\phi_X(p_n)}{p_n} \leq
\frac{C_PA}{a}\log\frac1{p_n}.
\]

Set \(t_n=1/p_n\). By (\ref{eq:testp}), \(t_n\uparrow\infty\) and
\[
\log t_{n+1}\leq\Lambda\log t_n.
\]
If \(t_n\leq t<t_{n+1}\), the monotonicity of \(U_X\) gives
\[
U_X(t) \leq U_X(t_{n+1}) \leq \frac{C_PA}{a}\log t_{n+1} \leq
\frac{C_PA\Lambda}{a}\log t.
\]
This proves (\ref{eq:testU}) for \(t\geq t_1\). For \(2\leq t<t_1\),
the same estimate follows, after changing the implicit constant,
from the monotonicity of \(U_X\).
\end{proof}

\begin{remark}
\label{rem:PIB-test} Lemma~\ref{lem:test} remains valid under the
testing \((X,L^1)\)-Poincar\'e inequality (\ref{eq:PI-B}), provided
that the functions \(f_n\) are bounded and that the gradient
condition is replaced by
\[
r_{B_n}\frac1{\mu(B_n)} \int_{\sigma B_n}g_n\,d\mu \leq A
p_n\log\frac1{p_n}.
\]
Indeed, the proof uses the Poincar\'e inequality only for the
functions \(f_n\).
\end{remark}

The second criterion uses a single radial cutoff. It applies when
the small-ball function has power-type behavior and sufficient
regularity.

\begin{lemma}[Small-ball criterion]
\label{lem:smallball} Assume that \((\Omega,d,\mu)\) supports the
testing \((X,L^1)\)-Poincar\'e inequality (\ref{eq:PI-B}). Suppose
that, for some \(x\in\Omega\), the function
\[
\psi(r)=-\log\mu(B(x,r))
\]
is convex for all sufficiently small \(r>0\), and that
\[
\psi(r)\simeq r^{-\beta} \qquad\text{as }r\downarrow0
\]
for some \(\beta>0\). Then
\[
U_X(t)\preceq\log t, \qquad t\geq2.
\]
\end{lemma}

\begin{proof}
Choose \(c_1,c_2>0\) and \(\lambda>1\) such that, for all
sufficiently small \(r\),
\[
c_1r^{-\beta}\leq\psi(r)\leq c_2r^{-\beta}, \qquad
\frac{c_2}{c_1}\lambda^{-\beta}\leq\frac12.
\]
Then
\[
\psi(\lambda r) \leq c_2\lambda^{-\beta}r^{-\beta} \leq\frac12
c_1r^{-\beta} \leq\frac12\psi(r).
\]
Set
\[
p(r)=\frac{\mu(B(x,r))}{\mu(B(x,\lambda r))}.
\]
Since
\[
\log\frac1{p(r)} = \psi(r)-\psi(\lambda r),
\]
we obtain
\[
\frac12\psi(r) \leq \log\frac1{p(r)} \leq \psi(r).
\]

Let us write
\[
\delta_r=\frac{r}{\psi(r)}.
\]
Since \(\psi(r)\to\infty\), we have \(\delta_r=o(r)\), and hence
\(r+\delta_r<\lambda r\) for sufficiently small \(r\). By the
convexity of \(\psi\), comparing the secant slopes on \([r/2,r]\)
and \([r,r+\delta_r]\), and using \(\delta_r/r=1/\psi(r)\), we
obtain
\[
\begin{aligned}
0
&\leq \psi(r)-\psi(r+\delta_r) \\
&\leq
\frac{2\delta_r}{r}\bigl(\psi(r/2)-\psi(r)\bigr) \\
&\leq \frac{2}{\psi(r)}\,\psi(r/2) \leq 2^{\beta+1}\frac{c_2}{c_1}
=:C_0.
\end{aligned}
\]
Therefore
\[
\mu(B(x,r+\delta_r)) \leq e^{C_0}\mu(B(x,r)).
\]

Let \(B=B(x,\lambda r)\), and consider the radial cutoff
\[
f(y)=\left(1-\frac{(d(x,y)-r)_+}{\delta_r}\right)_+.
\]
Then \(f=1\) on \(B(x,r)\), \(f=0\) outside \(B(x,r+\delta_r)\), and
the standard radial cutoff estimate shows that
\[
g=\delta_r^{-1}\chi_{B(x,r+\delta_r)}
\]
is an upper gradient of \(f\). Moreover,
\[
f_{B,\mu} \leq \frac{\mu(B(x,r+\delta_r))}{\mu(B(x,\lambda r))} \leq
e^{C_0}p(r).
\]
Since
\[
\log\frac1{p(r)} \geq \frac12\psi(r) \longrightarrow\infty
\qquad\text{as }r\downarrow0,
\]
we have \(e^{C_0}p(r)\leq1/2\) for sufficiently small \(r\).
Therefore
\[
|f-f_{B,\mu}|\geq\frac12 \qquad\text{on }B(x,r).
\]

Applying the testing \((X,L^1)\)-Poincar\'e inequality gives
\[
\frac12\phi_X(p(r)) \leq C_P\lambda r\, \frac{\mu(B(x,r+\delta_r))}
{\delta_r\,\mu(B(x,\lambda r))} \leq C_P\lambda e^{C_0}\psi(r)p(r).
\]
Hence
\[
U_X(1/p(r)) = \frac{\phi_X(p(r))}{p(r)} \leq 2C_P\lambda
e^{C_0}\psi(r) \leq 4C_P\lambda e^{C_0}\log\frac1{p(r)}.
\]

Finally, by assumption, \(\psi\) is finite and convex for all
sufficiently small \(r\), and hence it is continuous there.
Therefore
\[
r\longmapsto \log\frac1{p(r)} = \psi(r)-\psi(\lambda r)
\]
is continuous. Moreover,
\[
\log\frac1{p(r)} \simeq \psi(r) \longrightarrow\infty \qquad\text{as
}r\downarrow0.
\]
Thus it assumes every sufficiently large value. Consequently,
\[
U_X(t)\preceq\log t
\]
for all sufficiently large \(t\), and the remaining range \(t\geq2\)
is absorbed by changing the implicit constant.
\end{proof}

\subsection{Infinite product models}
\label{sec:products}

We consider infinite products of one-dimensional log-concave
probabilities equipped with weighted supremum metrics. Isoperimetric
inequalities for infinite products with supremum metrics have been
studied in \cite{BobkovHoudre,BartheSup,BartheHuou} and the
references therein. We recall the elementary properties of the
product models needed below, and then establish a uniform local
\(L\log L\)-Poincar\'e inequality on their balls. We also show that
any r.i.\ space \(X\) for which the corresponding
\((X,L^1)\)-Poincar\'e inequality holds must satisfy
\(U_X(t)\preceq\log t\).

Let \(\nu\) be a nondegenerate log-concave probability measure on
\(\mathbb R\), set \(S=\operatorname{supp}\nu\), and fix
\(b\in\operatorname{int}S\). Let \((a_j)_{j\geq1}\) be a
nonincreasing sequence of positive numbers with \(a_j\to0\), and
assume that
\begin{equation}
\label{eq:producttail} \sum_{j=1}^{\infty}
\nu\left\{|t-b|>\frac{\varepsilon}{a_j}\right\}<\infty
\qquad\text{for every }\varepsilon>0.
\end{equation}
This condition is automatic when \(S\) is compact. If \(S\) is
unbounded, a nondegenerate one-dimensional log-concave probability
has exponential tails \cite[Lemma~3.1]{Bobkov}. Hence
(\ref{eq:producttail}) holds whenever \(a_j\log j\to0\), since for
every \(\varepsilon>0\),
\[
\nu\left\{|t-b|>\frac{\varepsilon}{a_j}\right\} \leq
C\exp\left(-\frac{c\varepsilon}{a_j}\right)
\]
for all sufficiently large \(j\).%
\footnote{This includes, for example, geometrically decreasing
weights \(a_{j+1}\leq qa_j\), \(0<q<1\).}

Define
\[
\Omega_{a,b} = \{x\in S^{\mathbb N}:a_j|x_j-b|\to0\},
\]
equip it with the metric
\begin{equation*}
 d_a(x,y)=\sup_{j\geq1}a_j|x_j-y_j|,
\end{equation*}
and let
\[
\mu=\nu^{\otimes\mathbb N}
\]
denote the infinite product measure.

We first verify the standing assumptions. Applying Borel--Cantelli
to (\ref{eq:producttail}) with \(\varepsilon=1/k\), \(k\geq1\),
gives
\[
\mu(\Omega_{a,b})=1.
\]
Moreover, recalling that \(S=\operatorname{supp}\nu\), the map
\[
x\longmapsto (a_j(x_j-b))_{j\geq1}
\]
identifies \(\Omega_{a,b}\) isometrically with the closed subset
\[
c_0\cap\prod_{j\geq1}a_j(S-b)
\]
of \(c_0\). Hence \(\Omega_{a,b}\) is complete and separable, and
its Borel \(\sigma\)-algebra is generated by the coordinate maps.

Every ball has positive measure. Indeed, let \(x\in\Omega_{a,b}\)
and \(r>0\). Since \(a_j|x_j-b|\to0\), we have
\[
a_j|x_j-b|<\frac r2
\]
for all sufficiently large \(j\). For such \(j\),
\[
\left\{|t-b|<\frac{r}{2a_j}\right\} \subset \left\{|t-x_j|<\frac
r{a_j}\right\}.
\]
By (\ref{eq:producttail}), the complements of the sets on the left
have summable \(\nu\)-measure, and hence the corresponding infinite
product of probabilities is positive after discarding finitely many
indices. The remaining finitely many coordinate intervals also have
positive \(\nu\)-measure because \(x_j\in S\). Since
\(a_j|x_j-y_j|\to0\) for \(x,y\in\Omega_{a,b}\), the conditions
\[
a_j|x_j-y_j|<r \qquad\text{for every }j
\]
are equivalent to \(d_a(x,y)<r\). Therefore
\[
\mu(B(x,r))>0.
\]
Since \(\mu\) is a probability and \(\nu\) is nonatomic, all the
standing assumptions are satisfied.

We next observe that every Borel set of full \(\mu\)-measure has
infinite Hausdorff dimension. Let \(F\subset\Omega_{a,b}\) be Borel
with \(\mu(F)=1\), and, for \(m\geq1\), let
\[
\pi_m(x)=(x_1,\ldots,x_m).
\]
Since \(\Omega_{a,b}\) is a Polish space, the Borel probability
measure \(\mu\) is Radon. Hence, by inner regularity, there exists a
compact set \(K\subset F\) with \(\mu(K)>0\). Since \(\pi_m(K)\) is
compact,
\[
\nu^{\otimes m}(\pi_m(K)) = \mu\bigl(\pi_m^{-1}(\pi_m(K))\bigr) \geq
\mu(K)>0.
\]
As \(\nu^{\otimes m}\) is absolutely continuous with respect to
Lebesgue measure, \(\pi_m(K)\) has positive Lebesgue measure in
\(\mathbb R^m\). Therefore
\[
\dim_H\pi_m(K)=m.
\]
Moreover,
\[
|\pi_m(x)-\pi_m(y)| \leq \left(\sum_{j=1}^m
a_j^{-2}\right)^{1/2}d_a(x,y),
\]
so \(\pi_m\) is Lipschitz. Hence
\[
\dim_HF \geq \dim_HK \geq \dim_H\pi_m(K) = m.
\]
Since this holds for every \(m\geq1\), we obtain
\[
\dim_HF=\infty.
\]
As \(F\) was arbitrary, it follows that
\[
\dim_H^+\mu=\infty.
\]

To prove the endpoint inequality, we use the following normed
version of Proposition~\ref{prop:logPI}. Its proof follows the same
arguments, starting from Bobkov's dimension-free isoperimetric
inequality \cite[Proposition~2.1]{Bobkov}.

\begin{lemma}[Normed logarithmic endpoint]
\label{lem:normed-logPI} Let \(K\subset\mathbb R^m\) be a bounded
symmetric convex body with nonempty interior. Write \(\|\cdot\|_K\)
for its Minkowski norm and \(\|\cdot\|_{K,*}\) for the dual norm. If
\(\rho\) is an absolutely continuous log-concave probability measure
concentrated in \(x_0+rK\), then
\begin{equation}
\label{eq:normed-logPI} \|u-u_\rho\|_{L^{\Phi_0}(\rho)} \preceq
r\int g\,d\rho, \qquad u_\rho=\int u\,d\rho,
\end{equation}
for every \(u\in L^1(\rho)\) and every integrable upper gradient
\(g\) with respect to \(\|\cdot\|_K\). It is enough that \(u\) and
\(g\) be defined on a convex Borel set \(A\) of full
\(\rho\)-measure and satisfy the upper-gradient inequality along
curves contained in \(A\). The implicit constant is absolute and
independent of \(m\) and \(K\).
\end{lemma}

\begin{proof}
Applying \cite[Proposition~2.1]{Bobkov} with the symmetric convex
body \(rK\), and using \(\|\xi\|_{rK,*}=r\|\xi\|_{K,*}\), gives
\[
2r\int\|\nabla h\|_{K,*}\,d\rho \geq \operatorname{Ent}_\rho(h)
+\operatorname{Ent}_\rho(1-h)
\]
for every locally Lipschitz function \(h\) with values in \([0,1]\),
where
\[
\operatorname{Ent}_\rho(h) =\int h\log h\,d\rho -\left(\int
h\,d\rho\right) \log\left(\int h\,d\rho\right).
\]

Approximating characteristic functions by Lipschitz functions with
respect to \(\|\cdot\|_K\), as in \cite[Corollary~2.2]{Bobkov},
yields
\[
2r\,I_\rho(s)\geq s\log\frac1s+(1-s)\log\frac1{1-s}, \qquad
0<s\leq\frac12,
\]
where \(I_\rho\) is the isoperimetric profile associated with the
metric induced by \(\|\cdot\|_K\).

We now follow the argument of Proposition~\ref{prop:logPI}. First,
the coarea and layer-cake argument, with \(|\nabla u|\) replaced by
\(\|\nabla u\|_{K,*}\), gives
\[
\|u-u_\rho\|_{L^{\Phi_0}(\rho)} \preceq r\int\|\nabla
u\|_{K,*}\,d\rho
\]
for Lipschitz functions \(u\).

Second, let \(u\in L^1(\rho)\) have an integrable upper gradient
\(g\) on a convex Borel set \(A\) of full \(\rho\)-measure. Put
\(D=\operatorname{int}(\operatorname{supp}\rho)\). The density of
\(\rho\) is positive on \(D\), so \(A\cap D\) has full Lebesgue
measure in \(D\). Since \(A\) is convex, this implies \(D\subset
A\).

The line-segment argument gives \(u\in W^{1,1}_{\rm loc}(D)\).
Repeating it for a countable dense subset of \(\{v:\|v\|_K=1\}\), as
in Proposition~\ref{prop:logPI}, yields
\[
\|\nabla u\|_{K,*}\leq g \qquad\text{a.e. in }D.
\]
Choose an increasing sequence of compact convex sets \(C_j\Subset
D\) whose union is \(D\), with \(\rho(C_j)>0\), and set
\[
\rho_j=\frac{\rho|_{C_j}}{\rho(C_j)}.
\]
Each \(\rho_j\) is log-concave and remains concentrated in
\(x_0+rK\). The density of \(\rho\) is continuous and positive on
\(D\), and hence is bounded above and bounded away from zero on a
neighborhood of each \(C_j\).

Fix \(j\). By local mollification, there are smooth functions
\(u_\ell\) on a neighborhood of \(C_j\) such that
\[
u_\ell\to u,\qquad \nabla u_\ell\to\nabla u \quad\text{in
}L^1(C_j,\rho_j).
\]
Applying the Lipschitz estimate to \(\rho_j\), then passing to an
almost everywhere convergent subsequence and using the Fatou
property, gives
\[
\|u-u_{\rho_j}\|_{L^{\Phi_0}(\rho_j)} \leq Cr\int_{C_j}\|\nabla
u\|_{K,*}\,d\rho_j \leq Cr\int_{C_j}g\,d\rho_j,
\]
where \(C\) is the same absolute constant.

Finally, \(\rho(C_j)\to1\), \(u_{\rho_j}\to u_\rho\), and
\(\int_{C_j}g\,d\rho_j\to\int g\,d\rho\). Applying Fatou's lemma to
the Luxemburg modular, as in Proposition~\ref{prop:logPI}, proves
(\ref{eq:normed-logPI}).
\end{proof}

In the next theorem, we establish the sharp logarithmic endpoint for
these infinite product models. Since every ball has positive measure
and \(\dim_H^+\mu=\infty\), the measure \(\mu\) is necessarily
non-doubling.

\begin{theorem}
\label{thm:product} Let \((\Omega_{a,b},d_a,\mu)\) be the infinite
product model defined above. Then
\begin{equation}
\label{eq:productendpoint} \|f-f_{B,\mu}\|_{L^{\Phi_0}(B,\mu_B)}
\preceq r_B\int_B g\,d\mu_B
\end{equation}
for every ball \(B\), every \(f\in L^1(B,\mu_B)\), and every
integrable upper gradient \(g\) of \(f\) along rectifiable curves
contained in \(B\). The implicit constant is absolute.

Moreover, if \((\Omega_{a,b},d_a,\mu)\) supports the
\((X,L^1)\)-Poincar\'e inequality (\ref{eq:intro-PI}), then
\begin{equation}
\label{eq:productlog} U_X(t)\preceq\log t,\qquad t\geq2.
\end{equation}
Consequently, \((\Omega_{a,b},d_a,\mu)\) supports
(\ref{eq:intro-PI}) if and only if
\[
L\log L(0,1)\hookrightarrow X.
\]
\end{theorem}

\begin{proof}
Fix \(B=B_{d_a}(x,r)\), and set \(I_j=S\cap(x_j-r/a_j,x_j+r/a_j)\).
By the coordinate description of balls,
\(B=\Omega_{a,b}\cap\prod_{j\geq1}I_j\). Since
\(\mu(\Omega_{a,b})=1\) and \(\mu(B)>0\), the normalized restriction
factorizes as
\begin{equation*}
 \mu_B=\bigotimes_{j=1}^{\infty}\nu_j,
\qquad \nu_j=\frac{\nu|_{I_j}}{\nu(I_j)},
\end{equation*}
where each \(\nu_j\) is log-concave.

Let \(f\) and \(g\) satisfy (\ref{eq:productendpoint}). For
\(m\geq1\), write \(K_m=\prod_{j=1}^m I_j\),
\(\nu^{(m)}=\bigotimes_{j=1}^m\nu_j\), and
\(\nu_{>m}=\bigotimes_{j>m}\nu_j\). Equip \(\mathbb R^m\) with the
norm \(\|y\|_{a,m}=\max_{1\leq j\leq m}a_j|y_j|\). Then \(K_m\) lies
in the ball of radius \(r\) centered at \((x_1,\ldots,x_m)\).

For \(\nu_{>m}\)-almost every tail \(z\), we have \((y,z)\in B\) for
every \(y\in K_m\), and write \(f_z(y)=f(y,z)\) and
\(g_z(y)=g(y,z)\). By Fubini, these functions are integrable with
respect to \(\nu^{(m)}\) for almost every \(z\). Every rectifiable
curve in \(K_m\), with all coordinates \(j>m\) fixed at \(z\),
defines a curve in \(B\) of the same length, and hence \(g_z\) is an
upper gradient of \(f_z\).

Since \(\nu^{(m)}\) is a probability measure,
Lemma~\ref{lem:normed-logPI} gives
\[
\left\|f_z-\int_{K_m}f_z\,d\nu^{(m)}
\right\|_{L^{\Phi_0}(\nu^{(m)})} \leq Cr\int_{K_m}g_z\,d\nu^{(m)}.
\]
Let \(f_m\) be the conditional expectation of \(f\) with respect to
the first \(m\) coordinates, so that \(f_m(y)=\int
f(y,z)\,d\nu_{>m}(z)\). By Fubini's theorem and Minkowski's integral
inequality,
\[
\begin{aligned}
\|f_m-f_{B,\mu}\|_{L^{\Phi_0}(\nu^{(m)})} &\leq \int
\left\|f_z-\int_{K_m}f_z\,d\nu^{(m)}
\right\|_{L^{\Phi_0}(\nu^{(m)})}
\,d\nu_{>m}(z)\\
&\leq Cr\int_B g\,d\mu_B.
\end{aligned}
\]
Here the norm on the left-hand side coincides with
\(\|f_m-f_{B,\mu}\|_{L^{\Phi_0}(B,\mu_B)}\). Indeed, since the
coordinates generate the Borel \(\sigma\)-algebra, Doob's martingale
convergence theorem (see \cite{Kallenberg}) gives \(f_m\to f\)
almost everywhere, which, together with the Fatou property, yields
(\ref{eq:productendpoint}).

It remains to prove (\ref{eq:productlog}). Let \(w\) be the
continuous positive density of \(\nu\) on \(\operatorname{int}S\),
choose \(s_0>0\) with \([b-s_0,b+s_0]\subset\operatorname{int}S\),
and set
\begin{equation*}
 w_0=\min_{[b-s_0,b+s_0]}w>0,\quad
W=\max_{[b-s_0,b+s_0]}w<\infty, \quad v(s)=\nu((b-s,b+s)).
\end{equation*}
Thus \(v\) is continuous and strictly increasing on \([0,s_0]\),
with \(v(s)\geq2w_0s\). For \(0<s\leq s_0\) and \(0<\lambda\leq1\),
let \(t_\lambda(s)\in(0,s]\) be the unique number such that
\(v(t_\lambda(s))=\lambda v(s)\).

Fix \(0<c<s_0\) and \(0<\theta<1/2\), and write \(r_N=ca_N\),
\(s_{j,N}=r_N/a_j\), \(o=(b,b,\ldots)\), and \(B_N=B_{d_a}(o,r_N)\).
Since \((a_j)\) is nonincreasing, we have \(s_{j,N}\leq c\) for
\(j\leq N\), so we may define
\[
E_N=\{x\in B_N: |x_j-b|\leq t_\theta(s_{j,N}),\ 1\leq j\leq N\},
\]
Taking into account the product structure of \(\mu_{B_N}\), its
relative measure is given by
\begin{equation*}
 p_N=\mu_{B_N}(E_N) =\prod_{j=1}^N
\frac{v(t_\theta(s_{j,N}))}{v(s_{j,N})} =\theta^N.
\end{equation*}

Now put \(\eta=\theta w_0/(2W)\), \(h_N=\eta r_N/N\), and define
\[
\rho_N(x)=\max_{1\leq j\leq N}
a_j\bigl(|x_j-b|-t_\theta(s_{j,N})\bigr)_+, \qquad
f_N=(1-\rho_N/h_N)_+.
\]

Since \(\rho_N\) is \(1\)-Lipschitz with respect to \(d_a\), the
function \(f_N\) is \(1/h_N\)-Lipschitz, and \(\rho_N=0\) on
\(E_N\), so that \(f_N=1\) there. To estimate the measure of its
support, we use the density bounds, which give, for \(0<s\leq c\),
\[
s-t_\theta(s)\geq\frac{(1-\theta)w_0}{W}s
>\frac{\eta s}{N}.
\]

Consequently, the enlarged intervals remain inside \((b-s,b+s)\),
which implies
\begin{equation}
\label{eq:productenlarge}
\begin{aligned}
v(t_\theta(s)+\eta s/N)
&\leq\theta v(s)+\frac{2W\eta s}{N}\\
&\leq\theta v(s)(1+1/(2N)).
\end{aligned}
\end{equation}
Taking into account that \(h_N/a_j=\eta s_{j,N}/N\), the product
structure of \(\mu_{\sigma B_N}\) and (\ref{eq:productenlarge})
give, for every \(\sigma\geq1\),
\[
\begin{aligned}
\mu_{\sigma B_N}\{\rho_N<h_N\} &=\prod_{j=1}^N
\frac{v(t_\theta(s_{j,N})+h_N/a_j)}
{v(\sigma s_{j,N})}\\
&\leq\theta^N(1+1/(2N))^N \leq e^{1/2}p_N.
\end{aligned}
\]

Since \(0\leq f_N\leq\chi_{\{\rho_N<h_N\}}\), we have
\((f_N)_{B_N,\mu}\leq e^{1/2}p_N\to0\). Combining this with
\(f_N=1\) on \(E_N\), we conclude that, for all sufficiently large
\(N\),
\begin{equation}
\label{eq:productosc} |f_N-(f_N)_{B_N,\mu}|
\geq\tfrac12\qquad\text{on }E_N.
\end{equation}

To apply Lemma~\ref{lem:test}, we need to estimate the integral of
an upper gradient of \(f_N\). We claim that
\[
g_N=h_N^{-1}\chi_{\{\rho_N<h_N\}}
\]
is an upper gradient of \(f_N\). Indeed, since \(\rho_N\) is
\(1\)-Lipschitz, along any arc-length parametrized curve the
derivative of \(f_N\) is bounded by \(1/h_N\) on \(\{\rho_N<h_N\}\)
and vanishes almost everywhere on its complement, including the
boundary level. Using the preceding estimate for the measure of
\(\{\rho_N<h_N\}\), we obtain
\begin{equation}
\label{eq:productgrad}
\begin{aligned}
r_{B_N}\int_{\sigma B_N}g_N\,d\mu_{\sigma B_N}
&\leq\frac{e^{1/2}}{\eta}Np_N\\
&=\frac{e^{1/2}}{\eta\log(1/\theta)} p_N\log\frac1{p_N}.
\end{aligned}
\end{equation}
Since \(p_N=\theta^N\) decreases to zero and
\(p_N/p_{N+1}=1/\theta\), the sets \(E_N\), the functions \(f_N\),
and their upper gradients \(g_N\) satisfy the assumptions of
Lemma~\ref{lem:test}, by (\ref{eq:productosc}) and
(\ref{eq:productgrad}), for all sufficiently large \(N\). This
proves (\ref{eq:productlog}).

It remains to establish the final equivalence. If
(\ref{eq:intro-PI}) holds, then (\ref{eq:productlog}) and
Lemma~\ref{lem:LlogL-X} imply
\[
L\log L(0,1)\hookrightarrow X.
\]
Conversely, if \(L\log L(0,1)\hookrightarrow X\), then
(\ref{eq:LlogL-X-embedding}) and (\ref{eq:productendpoint}) give
\[
\|f-f_{B,\mu}\|_{X(B,\mu_B)} \preceq r_B\int_Bg\,d\mu_B.
\]
Thus (\ref{eq:intro-PI}) holds with \(\sigma=1\), which completes
the proof.

\end{proof}

Taking \(\nu\) to be normalized Lebesgue measure on \([0,1]\) gives
the Hilbert cube \([0,1]^{\mathbb N}\) with the weighted supremum
metric \(d_a\), for every nonincreasing sequence of positive weights
\(a_j\to0\), for instance \(a_j=2^{-j}\). Theorem~\ref{thm:product}
also applies to products of the standard Gaussian law when
\(a_j^2\log j\to0\), and of the exponential law on \([0,\infty)\)
when \(a_j\log j\to0\), choosing \(b>0\) in the latter case.

The following example shows that the logarithmic endpoint can remain
sharp even when \(\underline{\operatorname{ord}}_{\rm
loc}\mu(x)=\infty\) at every point. Thus, finiteness of the lower
local order is not necessary for the logarithmic barrier to hold.

\begin{example}
\label{ex:product-log} Let \(\nu\) be normalized Lebesgue measure on
\([0,1]\), let \(b=1/2\), and take \(a_j=1/\log(j+1)\). By
Theorem~\ref{thm:product}, the Hilbert cube \(([0,1]^{\mathbb
N},d_a,\mu)\) satisfies the \(L\log L\) endpoint inequality
(\ref{eq:productendpoint}), and every r.i.\ space \(X\) for which it
supports (\ref{eq:intro-PI}) satisfies \(U_X(t)\preceq\log t\).

On the other hand, the lower local order is infinite at every point.
Indeed, let \(x\in[0,1]^{\mathbb N}\) and \(0<r<1/8\). If
\(a_j>4r\), the interval \(\{t\in[0,1]:|t-x_j|<r/a_j\}\) has length
at most \(2r/a_j<1/2\). Since \(a_j>4r\) for \(1\leq
j<e^{1/(4r)}-1\), the product structure of the balls gives
\[
\psi_x(r)\geq \bigl(e^{1/(4r)}-2\bigr)\log2, \qquad
\psi_x(r)=-\log\mu(B(x,r)).
\]
Consequently, \(r^\beta\psi_x(r)\to\infty\) as \(r\downarrow0\) for
every \(\beta>0\), and hence
\[
\underline{\operatorname{ord}}_{\rm loc}\mu(x)=\infty
\qquad\text{for every }x\in[0,1]^{\mathbb N}.
\]
\end{example}

\subsection{Gaussian measures on Hilbert spaces}
\label{sec:hilbert-gauss}

We first consider a separable Banach space \(E\) equipped with a
centered Gaussian Radon probability measure \(\gamma\) whose support
is infinite-dimensional. Replacing \(E\) by
\(\operatorname{supp}\gamma\), we may assume that \(\gamma\) has
full support.

In this setting, \(\dim_H^+\gamma=\infty\). Indeed, for every
\(m\geq1\), there exists a continuous linear map \(T:E\to\mathbb
R^m\) whose image measure is a nondegenerate Gaussian probability.
If \(F\subset E\) is Borel with \(\gamma(F)=1\), inner regularity
gives a compact set \(K\subset F\) with \(\gamma(K)>0\). Since the
image measure is absolutely continuous with respect to Lebesgue
measure, \(T(K)\) has positive Lebesgue measure and hence \(\dim_H
T(K)=m\). As \(T\) is Lipschitz, \(\dim_H F\geq m\) for every \(m\),
and therefore \(\dim_H^+\gamma=\infty\).

Let \(H_\gamma\subset E\) denote the Cameron--Martin space of
\(\gamma\), equipped with its norm \(|\cdot|_{H_\gamma}\). Recall
that \(H_\gamma\) is dense in \(E\) (see \cite{BogachevGaussian}).
We next recall a standard consequence of the Cameron--Martin
formula, including a short proof for completeness.

\begin{lemma}
\label{lem:CM} For every \(h\in H_\gamma\) and \(s>0\),
\[
\gamma(B(h,s)) \geq e^{-|h|_{H_\gamma}^2/2}\gamma(B(0,s)).
\]
\end{lemma}

\begin{proof}
By the Cameron--Martin formula (see \cite{BogachevGaussian}),
\[
\gamma(B(h,s)) =e^{-|h|_{H_\gamma}^2/2} \int_{B(0,s)}e^{-\hat
h}\,d\gamma,
\]
where \(\hat h\) is the measurable linear functional associated with
\(h\). Since \(\gamma\) and \(B(0,s)\) are symmetric, averaging the
integrand with its reflection gives
\[
\gamma(B(h,s)) =e^{-|h|_{H_\gamma}^2/2} \int_{B(0,s)}\cosh\hat
h\,d\gamma \geq e^{-|h|_{H_\gamma}^2/2}\gamma(B(0,s)).
\]
\end{proof}

The preceding lemma allows us to transfer estimates for the measure
of balls centered at the origin to balls centered at points of
\(H_\gamma\). Since \(H_\gamma\) is dense in \(E\), this provides a
way to apply our results on lower local order and obtain
restrictions on the growth of \(U_X\).

\begin{corollary}
\label{cor:Gauss-order} Suppose that, for some \(\beta>0\),
\[
-\log\gamma(B(0,r))=O(r^{-\beta}) \qquad\text{as }r\downarrow0.
\]
Then
\[
\underline{\operatorname{ord}}_{\rm loc}\gamma(h)<\infty
\qquad\text{for every }h\in H_\gamma.
\]
Consequently, if \((E,\|\cdot\|,\gamma)\) supports the testing
\((X,L^1)\)-Poincar\'e inequality (\ref{eq:PI-B}), then
\[
\frac{U_X(t)}{\log t}\not\longrightarrow\infty \qquad\text{as
}t\to\infty.
\]
\end{corollary}

\begin{proof}
By Lemma~\ref{lem:CM}, for every \(h\in H_\gamma\),
\[
-\log\gamma(B(h,r)) \leq \frac{|h|_{H_\gamma}^2}{2}
-\log\gamma(B(0,r)) =O(r^{-\beta}),
\]
which implies that the lower local order is finite at every point of
\(H_\gamma\).

Now suppose that \(U_X(t)/\log t\to\infty\). Since \(H_\gamma\) is
dense in \(E\), the argument in the proof of Theorem~\ref{thm:SB}
would imply that \(\gamma\) is doubling, contradicting
\(\dim_H^+\gamma=\infty\).
\end{proof}

We now specialize to the Hilbertian setting. Let \(H\) be a
separable Hilbert space and let \(\mu\) be a centered Gaussian
measure on \(H\) with injective covariance operator \(Q\). We equip
\(H\) with its norm distance. Then \(\operatorname{supp}\mu=H\), and
every ball has positive measure. If \(H\) is infinite-dimensional,
the preceding argument gives \(\dim_H^+\mu=\infty\), so that \(\mu\)
is non-doubling.

All upper gradients below are taken with respect to the ambient
Hilbert norm, rather than the Cameron--Martin norm.

We shall use the following standard properties of log-concave
measures, which apply in particular to restrictions to convex sets
and finite-dimensional projections (see \cite{Borell}). We include
their short proofs for completeness.

\begin{lemma}
\label{lem:LC-stability} Let \(\nu\) be a log-concave probability
measure on a separable Banach space \(E\).
\begin{enumerate}
\item[\rm (i)] If \(D\subset E\) is a convex Borel set
with \(\nu(D)>0\), then \(\nu_D=\nu|_D/\nu(D)\) is log-concave.

\item[\rm (ii)] If \(T:E\to F\) is a continuous
linear map into a Banach space \(F\), then the image measure of
\(\nu\) under \(T\) is log-concave.

\item[\rm (iii)] For every \(x_0\in E\), the function
\(r\mapsto\log\nu(B(x_0,r))\) is concave on
\(\{r>0:\nu(B(x_0,r))>0\}\).
\end{enumerate}
\end{lemma}

\begin{proof}
For (i), the convexity of \(D\) gives
\[
(1-t)(U\cap D)+t(V\cap D) \subset ((1-t)U+tV)\cap D,
\]
and the assertion follows from the log-concavity of \(\nu\) after
normalization. Similarly, (ii) follows by applying log-concavity to
the inverse images of Borel sets under \(T\). Finally, the inclusion
\[
(1-t)B(x_0,r)+tB(x_0,s) \subset B(x_0,(1-t)r+ts)
\]
and the log-concavity of \(\nu\) give (iii).
\end{proof}

We first establish the logarithmic endpoint, without any additional
assumption on the small-ball behavior of \(\mu\).

\begin{theorem}
\label{thm:HilbertGauss} Let \(H\) be a separable Hilbert space and
let \(\mu\) be a centered Gaussian probability measure on \(H\) with
injective covariance operator. Then, for every ball
\(B=B(x_0,r_B)\subset H\), every \(f\in L^1(B,\mu_B)\), and every
integrable upper gradient \(g\) of \(f\) along rectifiable curves
contained in \(B\), it holds that
\begin{equation}
\label{eq:HilbertGauss} \|f-f_{B,\mu}\|_{L^{\Phi_0}(B,\mu_B)}
\preceq r_B\int_Bg\,d\mu_B.
\end{equation}
The implicit constant is absolute. Consequently, if \(L\log
L(0,1)\hookrightarrow X\), then \((H,\|\cdot\|,\mu)\) supports
(\ref{eq:intro-PI}) with \(\sigma=1\).
\end{theorem}

\begin{proof}
The finite-dimensional case follows from
Proposition~\ref{prop:logPI}, so we assume that \(H\) is
infinite-dimensional. By truncation and the Fatou property, we may
also assume that \(f\) is bounded.

Let \(Q\) be the covariance operator of \(\mu\), and let \((e_k)\)
be an orthonormal basis of eigenvectors of \(Q\), with
\(Qe_k=\lambda_ke_k\), \(\lambda_k>0\). Since \(Q\) is trace class
(see \cite{BogachevGaussian}), we have
\(\tau_n=\sum_{k>n}\lambda_k\to0\). Let \(P_n\) be the orthogonal
projection onto \(H_n=\operatorname{span}\{e_1,\ldots,e_n\}\), and
put \(R_n=I-P_n\), so that \(P_nx\) and \(R_nx\) are independent
under \(\mu\).

Choose \(\varepsilon_n\downarrow0\) such that
\(\tau_n/\varepsilon_n^2\to0\). By Chebyshev's inequality,
\begin{equation}
\label{eq:HG-mn} m_n:=\mu\{|R_nx|<\varepsilon_n\} \longrightarrow1.
\end{equation}

Write \(r=r_B\), fix \(0<s<r\), and put \(\delta=(r-s)/3\) and
\(B_s=B(x_0,s)\). Since \(\varepsilon_n+|R_nx_0|\to0\), for all
sufficiently large \(n\), every \(y\in
D_n:=B_{H_n}(P_nx_0,s+\delta)\) and every \(z\in R_nH\) with
\(|z|<\varepsilon_n\) satisfy
\[
|y+z-x_0| \leq |y-P_nx_0|+|z|+|R_nx_0| <s+2\delta<r.
\]
Thus \(y+z\in B\).

Let \(\mu_n\) and \(\mu_n^\perp\) be the image measures of \(\mu\)
under \(P_n\) and \(R_n\), respectively, and let \(\rho_n\) be the
image of \(\mu_{B_s}\) under \(P_n\). By independence,
\(\rho_n=\kappa_n\mu_n\), where
\[
\kappa_n(y) =\frac{\mu_n^\perp\{z:y+z\in B_s\}}{\mu(B_s)}
\leq\frac1{\mu(B_s)}.
\]
Since Gaussian measures are log-concave,
Lemma~\ref{lem:LC-stability}, applied first to the restriction to
\(B_s\) and then to the projection \(P_n\), shows that \(\rho_n\) is
log-concave. Since \(\rho_n=\kappa_n\mu_n\) and \(\mu_n\) is a
nondegenerate Gaussian measure on \(H_n\), the measure \(\rho_n\) is
absolutely continuous, and its density is log-concave in the usual
finite-dimensional sense (see \cite{Borell} and
\cite[p.~1917]{Bobkov}). Moreover, \(\rho_n(D_n)=1\).

For \(y\in D_n\), define
\[
\begin{aligned}
\varphi_n(y) &=\frac1{m_n}\int_{\{|z|<\varepsilon_n\}}
f(y+z)\,d\mu_n^\perp(z),\\
G_n(y) &=\frac1{m_n}\int_{\{|z|<\varepsilon_n\}}
g(y+z)\,d\mu_n^\perp(z).
\end{aligned}
\]
Every rectifiable curve in \(D_n\), with the remaining coordinates
fixed at \(z\), defines a curve in \(B\) of the same length whenever
\(|z|<\varepsilon_n\). Integrating the upper-gradient inequality for
\(f\) and using Tonelli's theorem therefore shows that \(G_n\) is an
upper gradient of \(\varphi_n\).

Let \(\lambda_n\) be the image of
\[
\rho_n\otimes \left(m_n^{-1}
\chi_{\{|z|<\varepsilon_n\}}\mu_n^\perp\right)
\]
under the map \((y,z)\mapsto y+z\). By independence,
\begin{equation*}
 d\lambda_n=h_n\,d\mu, \qquad h_n(x)=\frac1{m_n}
\chi_{\{|R_nx|<\varepsilon_n\}} \kappa_n(P_nx)
\leq\frac1{m_n\mu(B_s)}.
\end{equation*}
For all sufficiently large \(n\), the function \(h_n\) vanishes
outside \(B\), and hence \(\int
G_n\,d\rho_n=\int_Bgh_n\,d\mu<\infty\).

Applying Lemma~\ref{lem:normed-logPI} with the Euclidean unit ball
as \(K\), the convex domain \(A=D_n\), and radius \(s+\delta\), to
the absolutely continuous log-concave probability measure \(\rho_n\)
and the upper gradient \(G_n\), gives
\begin{equation}
\label{eq:HG-step} \left\|\varphi_n-\int\varphi_n\,d\rho_n
\right\|_{L^{\Phi_0}(\rho_n)} \leq C(s+\delta)\int_Bgh_n\,d\mu,
\end{equation}
where \(C\) is an absolute constant.

We now pass to the limit. Let \(\bar f=f\chi_B\) and \(\mathcal
F_n=\sigma(P_n)\). For \(x\in B_s\),
\[
\varphi_n(P_nx) =\frac1{m_n} \mathbb E_\mu\left[ \bar
f\,\chi_{\{|R_n\cdot|<\varepsilon_n\}} \mid\mathcal F_n\right](x),
\]
while
\[
\kappa_n(P_nx) =\frac{\mathbb E_\mu[\chi_{B_s}\mid\mathcal F_n](x)}
{\mu(B_s)}.
\]
The \(\sigma\)-algebras \(\mathcal F_n\) increase and generate the
Borel \(\sigma\)-algebra of \(H\). Thus Doob's martingale
convergence theorem (see \cite{Kallenberg}), together with
(\ref{eq:HG-mn}) and the \(L^1\)-contractivity of conditional
expectations, gives \(\varphi_n\circ P_n\to f\) in \(L^1(B_s,\mu)\).

Similarly, \(h_n\to\chi_{B_s}/\mu(B_s)\) in \(\mu\)-measure. Since
\(m_n\to1\), the functions \(h_n\) are uniformly bounded for all
sufficiently large \(n\). The absolute continuity of the integral of
\(g\) therefore gives
\[
\int_Bgh_n\,d\mu \longrightarrow \frac1{\mu(B_s)}\int_{B_s}g\,d\mu.
\]
On the other hand, since \(\varphi_n\circ P_n\to f\) in
\(L^1(B_s,\mu)\) and \(\rho_n\) is the image of \(\mu_{B_s}\) under
\(P_n\), we also have \(\int\varphi_n\,d\rho_n\to f_{B_s,\mu}\).

Since \(\rho_n\) is the image of \(\mu_{B_s}\) under \(P_n\), the
norm on the left-hand side of (\ref{eq:HG-step}) coincides with
\[
\left\|\varphi_n\circ P_n- \int\varphi_n\,d\rho_n
\right\|_{L^{\Phi_0}(B_s,\mu_{B_s})}.
\]
Passing to an almost everywhere convergent subsequence and using the
Fatou property, we obtain
\[
\|f-f_{B_s,\mu}\|_{L^{\Phi_0}(B_s,\mu_{B_s})} \leq
C(s+\delta)\int_{B_s}g\,d\mu_{B_s}.
\]

Finally, let \(s\uparrow r\). Since \(\mu(B_s)\uparrow\mu(B)\) and
\(f_{B_s,\mu}\to f_{B,\mu}\), for every
\(\lambda>Cr\int_Bg\,d\mu_B\), the preceding estimate gives, for
\(s\) sufficiently close to \(r\),
\[
\int_{B_s} \Phi_0\left(\frac{|f-f_{B_s,\mu}|}{\lambda}\right) \,d\mu
\leq\mu(B_s).
\]
Applying Fatou's lemma and taking the infimum over \(\lambda\)
proves (\ref{eq:HilbertGauss}), while the final assertion follows
from (\ref{eq:LlogL-X-embedding}).
\end{proof}

\begin{remark}
The \(L^1\)-Poincar\'e inequality in the Hilbertian Gaussian setting
can also be obtained from the curvature-dimension approach, using
the entropy convexity result of \cite[Proposition~2.3]{ASZ} and
\cite[Theorem~1]{RajalaCV}. Since Hilbert balls are convex, the
inequality holds with the same ball on both sides (see
\cite[Remark~1(i)]{RajalaCV}). Theorem~\ref{thm:HilbertGauss}
strengthens the oscillation norm from \(L^1\) to \(L\log L\).
\end{remark}

Theorem~\ref{thm:HilbertGauss} establishes the logarithmic endpoint
for every nondegenerate Gaussian measure on a Hilbert space, without
any additional assumption on its small-ball behavior. Under
two-sided power-type small-ball estimates, this endpoint is also
sharp.

\begin{corollary}
\label{cor:HilbertGauss} Suppose, in addition, that
\(\psi(r)=-\log\mu(B(0,r))\simeq r^{-\beta}\) as \(r\downarrow0\),
for some \(\beta>0\). Then the following are equivalent:
\begin{enumerate}
\item[\rm (i)] \((H,\|\cdot\|,\mu)\) supports
(\ref{eq:intro-PI});
\item[\rm (ii)] \((H,\|\cdot\|,\mu)\) satisfies
(\ref{eq:PI-B});
\item[\rm (iii)] \(L\log L(0,1)\hookrightarrow X\).
\end{enumerate}
\end{corollary}

\begin{proof}
The implication {\rm (iii)}\(\Rightarrow\){\rm (i)} follows from
Theorem~\ref{thm:HilbertGauss}, while {\rm (i)}\(\Rightarrow\){\rm
(ii)} is immediate. For {\rm (ii)}\(\Rightarrow\){\rm (iii)}, by
Lemma~\ref{lem:LC-stability}(iii), the function \(\psi\) is convex,
so Lemma~\ref{lem:smallball} gives \(U_X(t)\preceq\log t\). The
conclusion now follows from Lemma~\ref{lem:LlogL-X}.
\end{proof}

\begin{example}
Let \(\mu\) be the law of standard Brownian motion regarded as a
Gaussian measure on \(H=L^2(0,1)\). Its covariance operator is
injective, and
\[
-\log\mu(B(0,r))\sim\frac1{8r^2}, \qquad r\downarrow0
\]
(see \cite{LiShao}). Thus Corollary~\ref{cor:HilbertGauss} applies
with \(\beta=2\), providing a Hilbertian example of infinite upper
Hausdorff dimension, distinct from the weighted supremum product
models, for which the \(L\log L\) endpoint is sharp for both
(\ref{eq:intro-PI}) and (\ref{eq:PI-B}).
\end{example}

\begin{remark}
Assume that \((H,\|\cdot\|,\mu)\) satisfies (\ref{eq:PI-B}). The
upper bound \(\psi(r)=O(r^{-\beta})\) is sufficient, by
Corollary~\ref{cor:Gauss-order}, to rule out \(U_X(t)/\log
t\to\infty\). The stronger two-sided estimate \(\psi(r)\simeq
r^{-\beta}\) yields \(U_X(t)\preceq\log t\) and hence the exact
characterization in Corollary~\ref{cor:HilbertGauss}. For example,
if the eigenvalues of the covariance operator satisfy
\(\lambda_k\simeq k^{-\alpha}\), \(\alpha>1\), then \(\psi(r)\simeq
r^{-2/(\alpha-1)}\) (see \cite{LiShao,KuelbsLi}).
\end{remark}

\subsection{Wiener measure with the uniform norm}
\label{sec:Wiener}

We finally consider Wiener measure on \(C_0([0,1])\), equipped with
the uniform norm. In this setting, the logarithmic endpoint remains
sharp, with upper gradients taken with respect to the uniform
metric.

\begin{corollary}
\label{cor:Wiener} Let \(\gamma_W\) be Wiener measure on
\(C_0([0,1])\), equipped with \(\|\cdot\|_\infty\). Then
\(\gamma_W\) is non-doubling, and the following are equivalent:
\begin{enumerate}
\item[\rm (i)]
\((C_0([0,1]),\|\cdot\|_\infty,\gamma_W)\) supports
(\ref{eq:intro-PI});
\item[\rm (ii)]
\((C_0([0,1]),\|\cdot\|_\infty,\gamma_W)\) satisfies
(\ref{eq:PI-B});
\item[\rm (iii)]
\(L\log L(0,1)\hookrightarrow X\).
\end{enumerate}
\end{corollary}

\begin{proof}
We first prove {\rm (iii)}\(\Rightarrow\){\rm (i)}. Let \(P_nx\) be
the piecewise linear interpolation of \(x\) at the dyadic points
\(k2^{-n}\), and put \(R_n=I-P_n\). Then \(P_n\) is a linear
projection of norm \(1\), and its range \(V_n\) is isometric to
\(\ell^\infty_{2^n}\), since the uniform norm of a piecewise linear
function is attained at a dyadic node.

For Brownian motion \(W\), the random elements \(P_nW\) and \(R_nW\)
are independent: on each dyadic interval the remainder is a Brownian
bridge, independent of the dyadic nodal values. Moreover, uniform
continuity of Brownian paths gives \(\|R_nW\|_\infty\to0\) almost
surely. We may therefore choose \(\varepsilon_n\downarrow0\) such
that
\[
\gamma_W\{\|R_nW\|_\infty<\varepsilon_n\} \longrightarrow1.
\]
For every \(x_0\in C_0([0,1])\), we also have
\(\|R_nx_0\|_\infty\to0\), while the increasing \(\sigma\)-algebras
\(\sigma(P_n)\) generate the Borel \(\sigma\)-algebra of
\(C_0([0,1])\).

We now follow the proof of Theorem~\ref{thm:HilbertGauss}, replacing
the orthogonal projections by \(P_n\). Fix \(B=B(x_0,r)\), choose
\(0<s<r\), and put \(\delta=(r-s)/3\) and \(B_s=B(x_0,s)\). For all
sufficiently large \(n\),
\[
\|y+z-x_0\|_\infty \leq\|y-P_nx_0\|_\infty
+\|z\|_\infty+\|R_nx_0\|_\infty <s+2\delta<r
\]
whenever \(y\in D_n:=B_{V_n}(P_nx_0,s+\delta)\) and
\(\|z\|_\infty<\varepsilon_n\).

Let \(\rho_n\) be the image of \((\gamma_W)_{B_s}\) under \(P_n\).
By Lemma~\ref{lem:LC-stability}, \(\rho_n\) is log-concave.
Moreover, independence gives \(\rho_n=\kappa_n\mu_n\), where
\(\mu_n\) is the image of \(\gamma_W\) under \(P_n\) and
\(0\leq\kappa_n\leq1/\gamma_W(B_s)\). Since \(\mu_n\) is a
nondegenerate Gaussian probability on \(V_n\), the measure
\(\rho_n\) is absolutely continuous and has a log-concave density
(see \cite{Borell} and \cite[p.~1917]{Bobkov}). The contractivity of
\(P_n\) also gives \(\rho_n(D_n)=1\).

Thus the finite-dimensional step in the proof of
Theorem~\ref{thm:HilbertGauss} applies with
Lemma~\ref{lem:normed-logPI} in place of
Proposition~\ref{prop:logPI}, using the \(\ell^\infty\)-norm on
\(V_n\). The constant is independent of \(\dim V_n=2^n\). The
averaging, martingale convergence and Fatou arguments are unchanged
and yield
\[
\|f-f_{B,\gamma_W}\|_ {L^{\Phi_0}(B,(\gamma_W)_B)} \preceq
r\int_Bg\,d(\gamma_W)_B.
\]
Together with (\ref{eq:LlogL-X-embedding}), this proves {\rm
(iii)}\(\Rightarrow\){\rm (i)}. The implication {\rm
(i)}\(\Rightarrow\){\rm (ii)} is immediate.

For {\rm (ii)}\(\Rightarrow\){\rm (iii)}, put
\(V(r)=\gamma_W(B(0,r))\) and \(\psi(r)=-\log V(r)\). The classical
exit-time expansion gives
\begin{equation*}
 V(r)=\frac4\pi\sum_{k=0}^{\infty}
\frac{(-1)^k}{2k+1} \exp\left(-\frac{(2k+1)^2\pi^2}{8r^2}\right)
\end{equation*}
(see \cite{BorodinSalminen,LiShao}). Consequently, with
\(a=\pi^2/8\),
\begin{equation}
\label{eq:Wienerasymp} V(r)=\frac4\pi e^{-a/r^2}
\left(1+O(e^{-8a/r^2})\right), \qquad r\downarrow0.
\end{equation}
In particular, \(V(2r)/V(r)\to\infty\), so \(\gamma_W\) is
non-doubling.

Since Wiener measure is Gaussian, Lemma~\ref{lem:LC-stability}(iii)
shows that \(\psi\) is convex, while (\ref{eq:Wienerasymp}) gives
\(\psi(r)\simeq r^{-2}\). Assuming {\rm (ii)},
Lemma~\ref{lem:smallball} therefore yields \(U_X(t)\preceq\log t\),
and Lemma~\ref{lem:LlogL-X} gives \(L\log L(0,1)\hookrightarrow X\).
\end{proof}

\section*{Declaration of generative AI and AI-assisted
technologies in the manuscript preparation process}

During the preparation of this manuscript, the author used ChatGPT
(OpenAI) and Claude (Anthropic) to assist with language editing,
manuscript organization, critical examination of mathematical
arguments, and bibliographic searches. The author reviewed and
revised the resulting suggestions and takes full responsibility for
the mathematical content, references, and final text.

\end{document}